\documentclass[11pt,a4paper]{article}
\usepackage[utf8]{inputenc}
\usepackage{graphicx}
\usepackage{subcaption}

\usepackage{amsmath,amssymb,amscd,amsthm,amsfonts}
\usepackage{amsmath,amsfonts,amssymb,amsthm,epsfig,epstopdf,titling,url,array, amscd}
\usepackage{dsfont}
\usepackage{tikz}
\usepackage[margin=3cm]{geometry}
\usepackage{floatrow}
\usepackage{hyperref}
\usepackage[nobysame, alphabetic]{amsrefs}
 \usepackage[neverdecrease]{paralist}

\theoremstyle{plain}
\newtheorem{theorem}{Theorem}[section]
\newtheorem{lemma}[theorem]{Lemma}

\theoremstyle{definition}

\newtheorem{remark}[theorem]{Remark}

\newcommand{\R}{\mathbb{R}}
\newcommand{\x}{\boldsymbol{x}}
\newcommand{\y}{\boldsymbol{y}}
\newcommand{\conv}{\mathop{\mathrm{conv}}}

\author{Isaac Arelio, Luis Montejano\thanks{Research  funded by PAPIIT grant AGI00721}, Deborah Oliveros }

\title{A 4-Dimensional Peabody of Constant Width}

\begin{document}

\maketitle

\begin{abstract}
In this paper we present a unique $4$-dimensional body of constant width based on the classical notion of focal conics. 
\end{abstract}

\section{Introduction}

In \cite{AMO}, we describe a family of $3$-dimensional bodies of constant width which we  called peabodies, obtained from the Reuleaux tetrahedron by replacing a small neighborhood of its six singular edges with sections of an envelope of $3$-balls. (See Figure \ref{patric}).

\begin{figure}[h]
\begin{center}
\includegraphics[scale=.25]{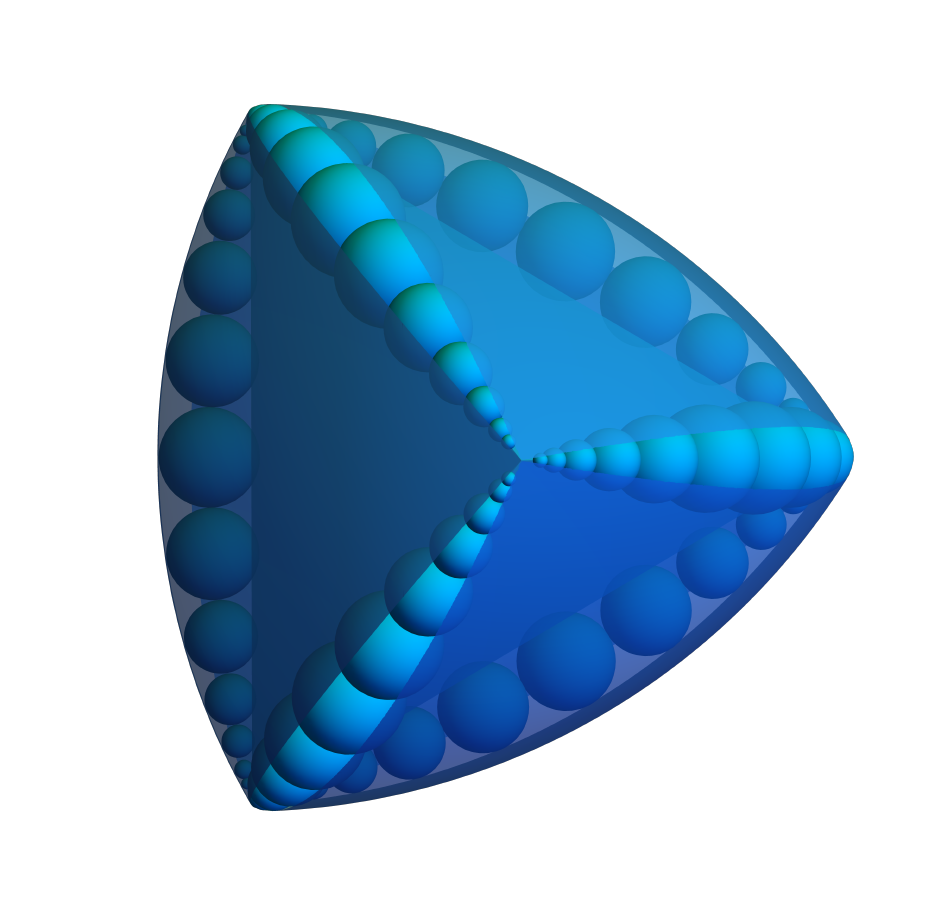}
\caption{A peabody of constant width }\label{patric}
\end{center}
\end{figure}

The purpose of this paper is to describe a $4$-dimensional body of constant width obtained from the $4$-simplex Reuleaux body  by replacing a small neighborhood of its $2$-skeleton with sections of a $4$-ball envelope. Besides the obvious rotational bodies of constant width, this is a unique concrete body of constant width that is known in dimension $4$.

Although in \cite{AMO}, we find a whole family of $3$-dimensional peabodies of constant width, everything seems to indicate that there is a unique $4$-dimensional peabody of constant width constructed in this manner. 
Indeed, this exceptional body of constant width circumscribes the $4$-simplex $\Delta^4$, has exactly its same symmetries and, with the exception of its $5$ vertices, has a smooth boundary.
 
 Behind the construction of the $3$-dimensional peabodies of constant width  lies the classical notion of focal conics discussed, for example, by Hilbert in his famous book \emph{Geometry and Imagination} \cite{H}.  In the $3$-dimensional case \cite{AMO}, we replace the singularities of the Reuleaux  tetrahedron with 
subarcs of conics in such a way that two dual edges are focal conics. In this paper, and for the construction of the $4$-dimensional peabody of constant width, we obtain a new and unique $2$-skeleton out of the regular $4$-simplex, where every edge is replaced by a subarc of an ellipse and every triangle is replaced by a piece of a hyperboloid of revolution. As incredible as it may seem,  
everything matches in such a way that the two dual ``faces" are focal quadrics. Furthermore, the symmetries of this new transformed  
version of the $2$-skeleton are exactly the symmetries of the $2$-skeleton of a regular $4$-simplex. 

This paper is organized as follows. In Section \ref{sec:focal} we introduce the notion of focal quadrics, a definition based on the classical notion of focal conics. In Section \ref{sec:stainer} we define the concept of elliptic and hyperbolic Steiner chains of $4$-balls.
   In Section \ref{sec:embededsimplices} we prove that the regular $4$-simplex is focally embedded in a pair of two focal quadrics and use the notions we have defined to prove
the existence of a pair of $3$-dimensional submanifolds in $\R^4$. This pair of submanifolds is related to a pair of dual faces of the $4$-regular simplex and are in a certain sense ``opposite," since in them the property of constant width is achieved. In Section \ref{sec:2-skeleton}, we show how a focal $2$-skeleton of  $\Delta^4$ is obtained, where each edge is replaced by a subarc of an ellipse, each triangle is replaced by a piece of a hyperboloid of revolution and  if an edge and a triangle of the $2$-skeleton are dual, then the corresponding  quadrics are focal.  In Section \ref{sec:assembly}, we show how all these submanifolds are assembled to yield the boundary of a $4$-dimensional convex body and finally, in Section \ref{sec:diameter}, we verify that  constant width is achieved.

\section{Focal Quadrics}\label{sec:focal}

Consider two conics lying in orthogonal planes, whose principal axes coincide with the intersection of the planes.  (See Figure \ref{cocon}). We say that these two conics are focal  if the focus of each of them lies on the other. See Section 2.4 of the  book \emph{Geometry and the Imagination} by Hilbert and Cohn-Vossen \cite{H}.

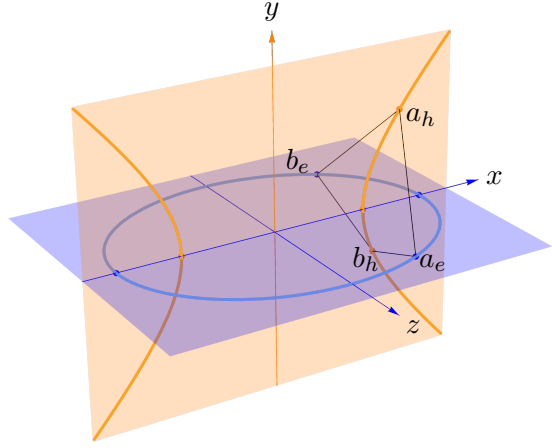
\begin{figure}[h!]
\begin{center}
\includegraphics[scale=.5]{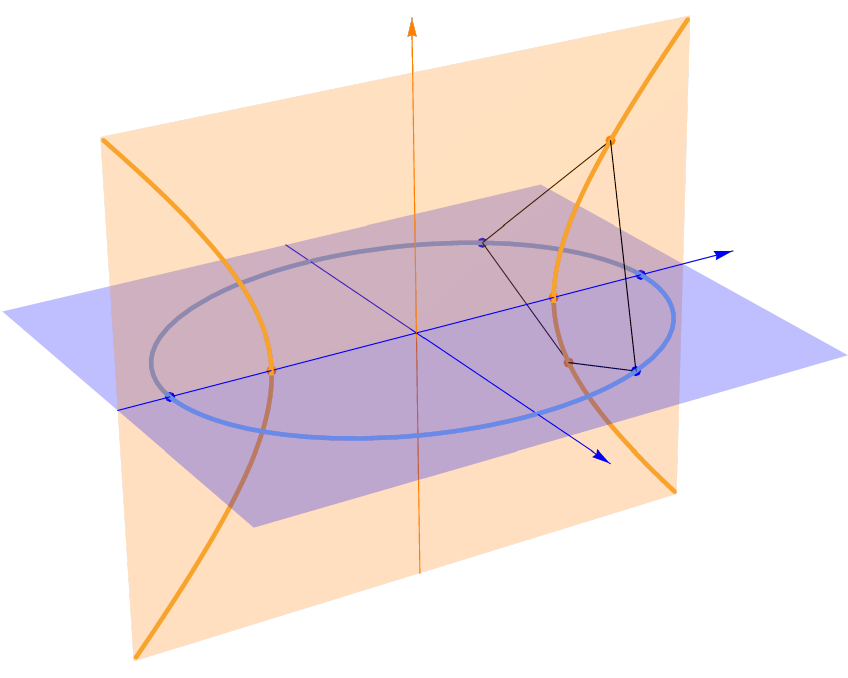}\put(-25,101){$x$}\put(-108,165){$y$}\put(-55,45){$z$}
\put(-50,70){$a_e$}\put(-55,125){$a_h$}\put(-75,70){$b_h$}\put(-100,107){$b_e$}
\caption{A pair of Focal conics sharing their principal axis as the $x$-axis. }\label{cocon}
\end{center}
\end{figure}

In the standard case, when we consider the euclidean $3$-space with coordinates $(x,y,z)$ one of the two focal conics is an ellipse and the other a hyperbola, sharing their principal axis, say the $x$-axis, then, their corresponding equations are given as follows, assuming that $a>b$.

\begin{equation}\label{ellipse-hyperbola}
\frac{z^2}{a^2-b^2}+\frac{x^2}{a^2}=1 \quad\quad \text{ and } \quad\quad \frac{y^2}{a^2-b^2}-\frac{x^2}{b^2}=1.	
\end{equation}

The following result was proved in \cite{AMO}, [\emph{Theorem 2.1a}] and is relevant to our construction. 
Throughout the paper we will abuse the notation by denoting by $pq$ both the interval whose endpoints are the points $p$ and $q$, and its length.

\begin{theorem}\label{unoyuno}
Let $E$ and $H$ be two focal conics, where $E$ is an ellipse and $H$ is an hyperbola. 
Suppose $a_e,b_e\in E$, and  $a_h,b_h\in H$.   If  $a_h, b_h$ lie in the same connected component of $H$, then
\[a_ea_h+b_eb_h=a_hb_e+a_eb_h.\]   
\end{theorem}

\noindent {\bf Remark:} \emph{The above theorem tells us that the whole ellipse plays the role of the foci of the hyperbola. 
This is so because any two points of the ellipse (not only the foci of the hyperbola) can be chosen to play the characteristic role of foci for the hyperbola.} 

\smallskip

Next, we will generalize this notion to focal quadrics for our needs in dimension $4$.  Let $\Gamma^3$ be a $3$-dimensional hyperplane in 
$\R^4$ and let $\Gamma^2$ be a $2$-dimensional plane in $\R^4$  intersecting orthogonally $\Gamma^3$ at a line $l$.  
Consider an ellipse $E\subset \Gamma^2$ with principal axis in $l$ and an hyperbola of revolution 
$\mathcal{H}\subset \Gamma^3$ also with axis in $l$, with the property that the focus of each of them lies on the other, 
thus for every  $2$-dimensional plane $\Gamma'\subset \R^4$ containing $l$, the ellipse $E$ and the hyperbola $\mathcal{H}\cap\Gamma'$ are focal conics in $\Gamma^2+ \Gamma'$. If this is the case, we shall say that the ellipse $E$ and the hyperboloid of revolution $\mathcal{H}$ are \emph{focal quadrics}.
(See Figure \ref{concon}).

In the standard case, when $l$ is the $x$-axis, if we assume that the coordinates of the points $\R^4$ are given $(x,y,z,w)$, the ellipse $E\subset \{x,z\}$-plane,
 and the hyperboloid of revolution around the $x$-axis is contained in the $\{x,y,w\}$-subspace of 
 $\R^4$. Thus, the equations are given as follows:   
\begin{equation}\label{ellipse-Rhyperbola}
\frac{z^2}{a^2-b^2}+\frac{x^2}{a^2}=1 \quad\quad \text{ and } \quad\quad \frac{y^2}{a^2-b^2}-\frac{x^2}{b^2}+\frac{w^2}{a^2-b^2}=1. 
\end{equation}

\begin{figure}[h!]
\begin{center}
\includegraphics[scale=.59]{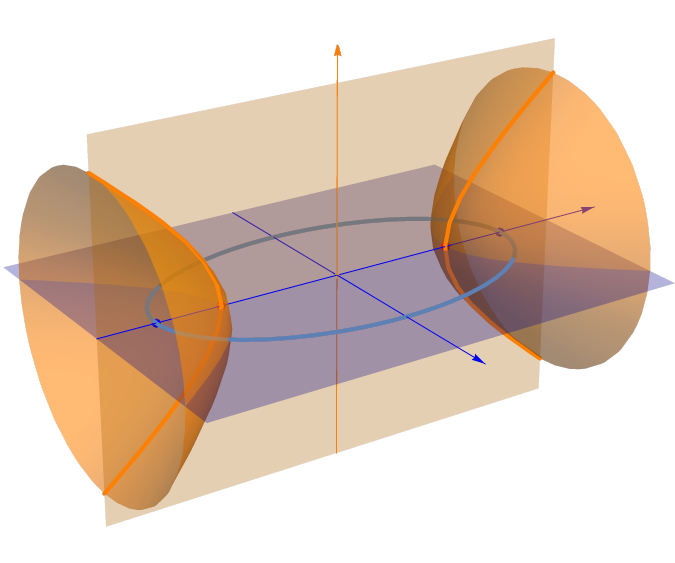}
\caption{A $3$-dimensional representation of focal quadrics in $\R^4$. Clearly, in dimension $4$, the ellipse $E$ and the hyperboloid of 
revolution do not intersect each other.}\label{concon}
\end{center}
\end{figure}

\begin{theorem}\label{Runoyuno}  Let $E$ be an ellipse  and $\mathcal{H}$ be a hyperboloid of revolution in $\R^4$.  Assume $E$ and $\mathcal{H}$ are focal quadrics. 
Suppose $a_e\in E$,  $a_h\in \mathcal{H}$,  $f_e\in l $ is a focus of $E$, and $f_h\in l$ a focus of $\mathcal{H}$.
If  $a_h$ and $f_e$ lie in the same connected component of $\mathcal{H}$, then
\[a_ea_h - a_hf_e - a_ef_h \ \ \text{is\ constant.}\]
 
\end{theorem}

\begin{proof}
 Let $\Gamma^2$ be the $2$-dimensional subspace of $\R^4$ containing $l$ and $a_ea_h$.  Then $\Gamma^2 \cap \mathcal{H}$ is a hyperbola $H$ which is focal with the ellipse $E$.  Furthermore,  $a_e\in E$,  $a_h\in H$, $f_h \in E$ and $f_e \in H$.    
Since  $a_h, f_e$ lie in the same component of $\Gamma \cap \mathcal{H}$, then by Theorem \ref{unoyuno},
$a_ea_h-a_hf_e - a_ef_h= -(f_hf_e)$ is constant. 
\end{proof}

Given the ellipse $E$ and the hyperbola $H$, focal conics, with common axis the line $l$ in $\R^3$, it was observed in \cite{AMO} that it is always possible to embed a regular tetrahedron $\Delta^3$ with vertices, say $\{p_1, p_2, p_3, p_4\}$, in such a way the following three conditions hold:
\begin{enumerate}[1)]
\item  the segment $p_1p_2$ is orthogonal to $l$ passes through the midpoint $p_{12}$ and  $\{p_1, p_2\}\subset E$,
\item the segment $p_3p_4$ is orthogonal to $l$ passes through the midpoint  $p_{34}$ and $\{p_3, p_4\}\subset H$, 
\item there is a subarc $H_{34}$ of the hyperbola $H$ joining $p_3$ with $p_4$ and a subarc $E_{12}$ of the ellipse $E$ joining $p_1$ with $p_2$, such that $H_{34}\cap \Delta_3=\{p_3,p_4\}$ and $E_{12}\cap \Delta_3=\{p_1,p_2\}$.
\end{enumerate}

\noindent In this case (after a permutation of the vertices),  we will say that \emph{$\Delta^3$ is focally embedded in $E,H$}. See Figure \ref{confo}.

\begin{center}   
\begin{figure}[h!]
\includegraphics[scale=.5]{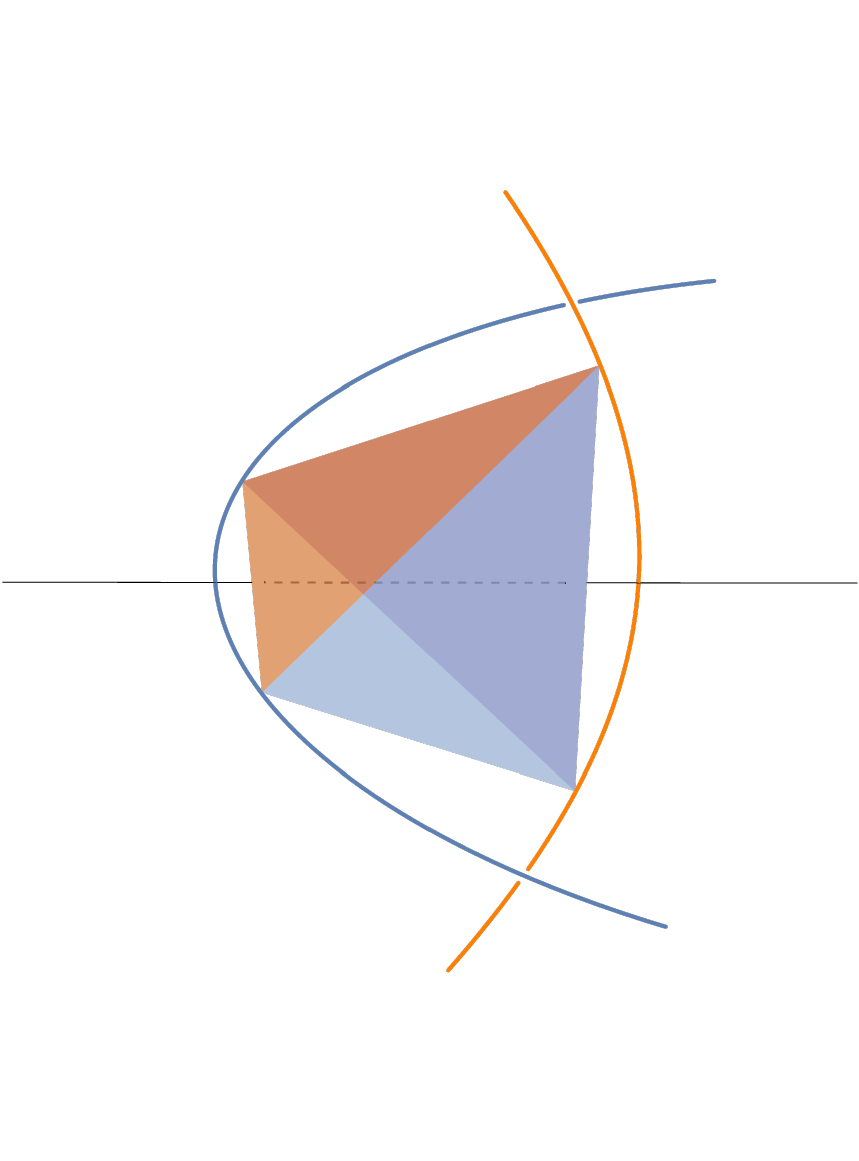}\put(-44,52){\textcolor{blue}{$E$}}\put(-100,235){\textcolor{orange}{$H$}}\put(6,137){$l$}
\put(-65,85){$p_4$}\put(-60,195){$p_3$}\put(-50,150){$H_{34}$} \put(-153,108){$p_1$} \put(-160,165){$p_2$}\put(-170,145){$E_{12}$}
\caption{Focal embedding of a tetrahedron in $E$ and $H$ in $\R^3$ }\label{confo}
\end{figure}
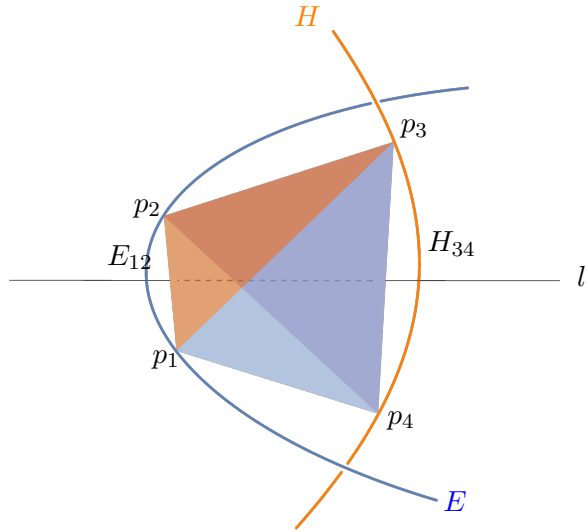
\end{center}

\section{Steiner Chains}\label{sec:stainer}

\subsection{Elliptic Steiner chains}\label{Elliptic stainer}

Let $C_1$ and $C_2$ be two circles in the plane with centers at $c_1$ and $c_2$, respectively, $c_1\not= c_2$. Suppose
$C_1\cap C_2=\{a,b\}$, where $ab$ is orthogonal to $c_1c_2$ and assume the line through $ab$ leaves $c_1$ and $c_2$ on the same side. Consider the collection of circles $\{C_{\alpha}\}$  tangent to both $C_1$ and $C_2$ inside their symmetric difference 
(Figure \ref{ESTCH}).  Then it is easy to prove that the centers of all these circles lie on the ellipse $E$ with foci at  $\{c_1, c_2\}$, containing the points $\{a,b\}$. 

Conversely, let an ellipse $E$ with foci $\{c_1, c_2\}$ in the plane and a chord $ab$ of $E$ orthogonal to $c_1c_2$ be given such that the line through $ab$ leaves $c_1$ and $c_2$ on the same side. Let $C_i$ be the circles centered at $c_i$ intersecting at both points  $\{a,b\}$.  Assume that, say, $c_2$ is the center that  is further away from $ab$ than $c_1$. Then \emph{the elliptic Steiner chain of disks, centered at $E$ and based in $ab$} is the collection of disks $\{D_{\alpha}\}$  tangent to both $C_1$ and $C_2$ inside $C_1\setminus C_2$. 

Assume the ellipse $E$ is lying in a plane $\Gamma^2\subset \mathbb{R}^4$. Given an elliptic Steiner chain of disks centered at $E$ and based in $ab$,  we will call (with a slight abuse of language) the \emph{elliptic Steiner chain of $4$-balls centered at $E$ and based in $ab$} denoted by $\Omega_E$  the collection of $4$-balls  whose intersection with $\Gamma^2$ is precisely the disks of the original Steiner chain. (See  \cite[pp.~51-54]{O} for more information about Steiner chains.)


\begin{figure}[h!]
\begin{center}
\includegraphics[scale=.4,angle=-90]{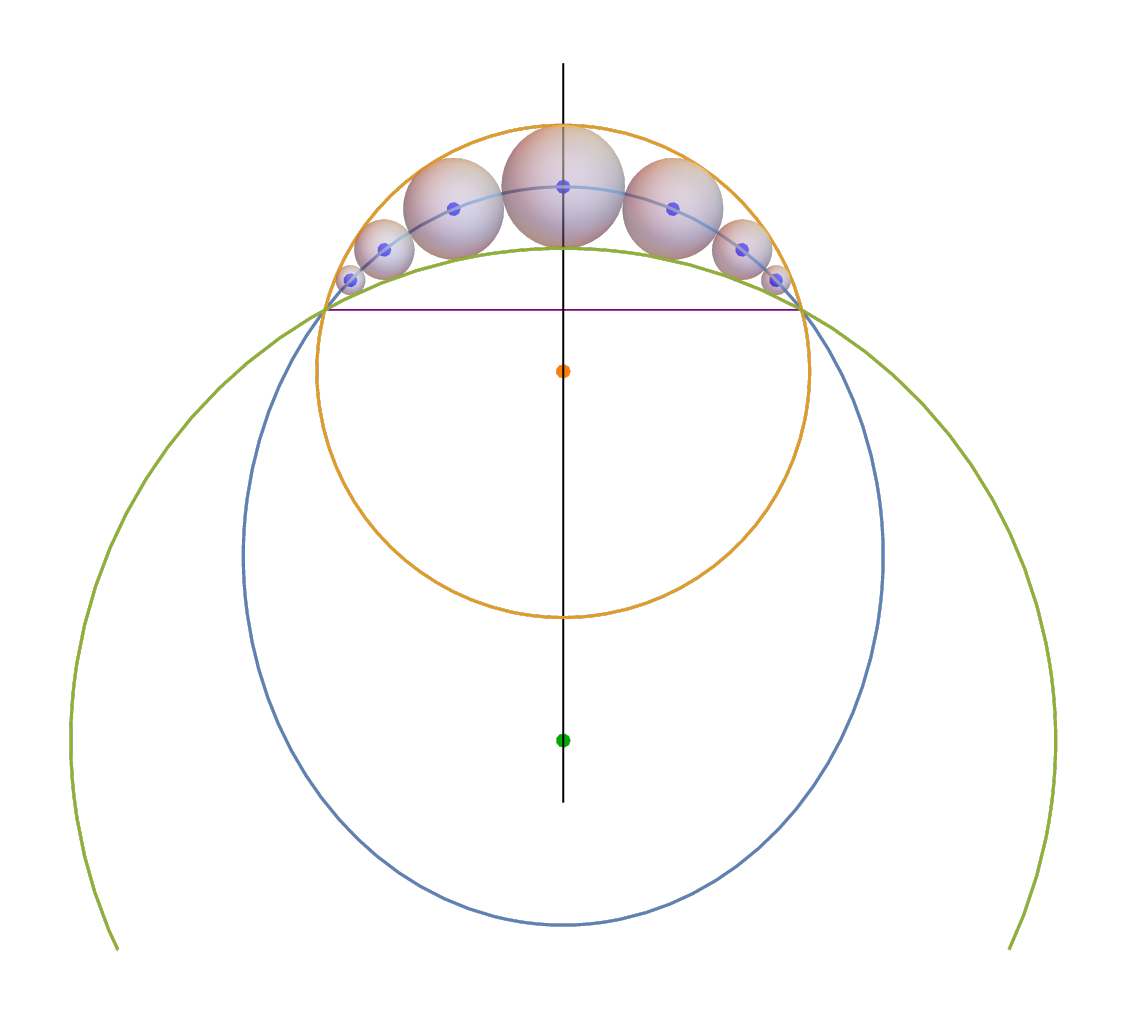}\put(-145,-100){\textcolor{olive}{$c_2$}}\put(-75,-100){\textcolor{orange}{$c_1$}}\put(-190,-110){\textcolor{blue}{$E$}}\put(-60,-60){$a$}\put(-60,-165){$b$}
\label{ESTCH}
\caption{Elliptic Steiner chain }\label{ESTCH}
\end{center}
\end{figure}
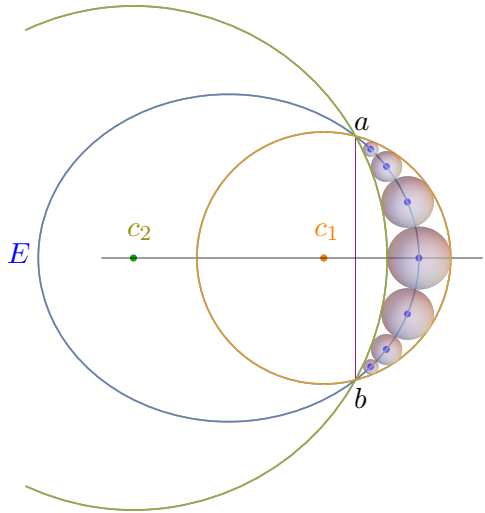

\subsection{Hyperbolic Steiner chains}\label{hiperbolicStainer}

Similarly, let $C_1$ and $C_2$ be two circles in the plane with centers at $c_1$ and $c_2$, respectively, $c_1\not= c_2$. Suppose
$C_1\cap C_2=\{a,b\}$, where $ab$ is orthogonal to $c_1c_2$ and assume the line through $ab$ separates $c_1$ from $c_2$. Consider the collection of circles $\{C_{\alpha}\}$  tangent to both $C_1$ and $C_2$ inside $C_1$ and $C_2$. Then it is easy to prove that the centers of all these circles lie on the hyperbola $H$ with foci at  $\{c_1, c_2\}$ (Figure \ref{HSTCH}).   

Conversely, let a hyperbola $H$ with foci $\{c_1, c_2\}$ in the plane and a chord $ab$ of $H$  orthogonal to $c_1c_2$ be given such that the line through $ab$ separates $c_1$ from $c_2$.  Let $C_i$ be the circles centered at $c_i$ intersecting at both points  $\{a,b\}$.  The collection of disks $\{D_{\alpha}\}$  tangent to both $C_1$ and $C_2$ inside $C_1$ and $C_2$ is called \emph{the hyperbolic Steiner chain of disks centered at $E$ and based in $ab$}.  

\begin{figure}[h!]
\begin{center}
\includegraphics[scale=.4]{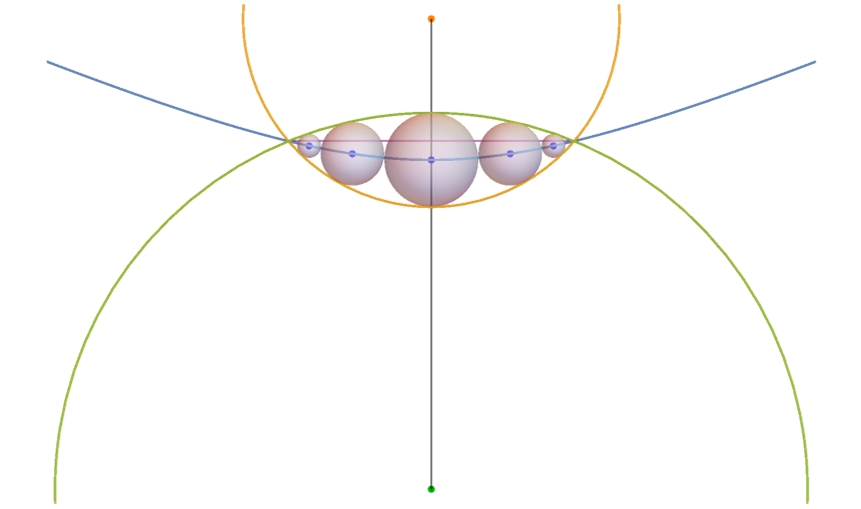}\put(-100,0){\textcolor{olive}{$c_1$}}\put(-100,115){\textcolor{orange}{$c_2$}}\put(0,100){\textcolor{blue}{$H$}}\put(-143,75){$a$}\put(-70,75){$b$}
\caption{Hyperbolic Steiner Chain }\label{HSTCH}
\end{center}
\end{figure}

Next, we will generalize this notion of the hyperbolic Steiner chain to our needs in dimension $4$. Let  $\mathbb{S}_1$ and $\mathbb{S}_2$ be two $2$-spheres in euclidean $3$-space with centers at $c_1$ and $c_2$ such that $\mathbb{S}_1\cap \mathbb{S}_2$ is the circle $C$, and such that the plane containing $C$ separates $c_1$ from $c_2$. Consider the collection of $2$-spheres $\{\delta_{\alpha}\}$  tangent to both $\mathbb{S}_1$ and $\mathbb{S}_2$ inside $\mathbb{S}_1$ and $\mathbb{S}_2$. Then it is not difficult to observe that the centers of all these spheres lie in the hyperboloid of revolution $\mathcal{H}$ with foci at  $\{c_1, c_2\}$, containing the circle $C$.   

Conversely, let a hyperboloid of revolution $\mathcal{H}$ be given with foci at $\{c_1, c_2\}$ in some $3$-space $\Gamma ^3$ and $C$ a circle contained in $\mathcal{H}$ such that the plane through $C$ is orthogonal to $c_1c_2$ and separates $c_1$ from $c_2$. Let $\mathbb{S}_i$ be the $2$-sphere centered at $c_i$, $i=\{1,2\}$ both containing the circle $C$. Then, \emph{the hyperbolic Steiner chain of 3-balls centered at $\mathcal{H}$ and based in $C$} is the collection of $3$-balls $\{B_{\alpha}\}$  tangent to both $\mathbb{S}_1$ and $\mathbb{S}_2$ inside $\mathbb{S}_1$ and $\mathbb{S}_2$ (see Figure \ref{3-hyperbolic}). 

Now assume that the hyperboloid of revolution $\mathcal{H}$ is lying on a $3$-space, say $\Gamma^3$, that is contained in $\mathbb{R}^4$. Then again by a slight abuse of language  we will call the collection of $4$-balls whose intersection with $\Gamma^3$ is precisely the $3$-balls of the original Steiner chain, \emph{the hyperbolic Steiner chain of $4$-balls} centered at $\mathcal{H}$, based in $C$ and denoted by $\Omega_{\mathcal{H}}$. 

\begin{figure}[h!]
\begin{center}
\includegraphics[scale=.36,angle=90]{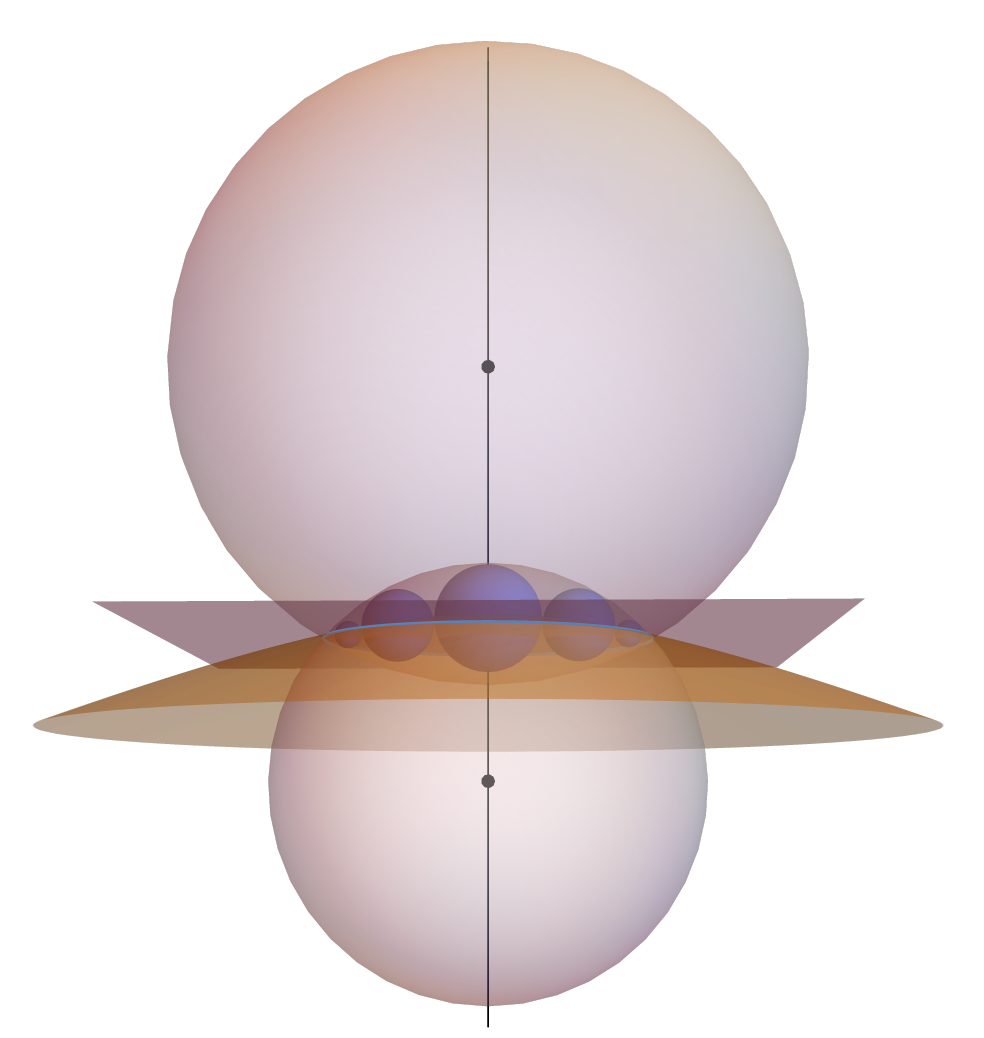}\put(-46,90){$c_2$}\put(-118,90){$c_1$}\put(-20,150){$\Gamma^3$}\put(-50,10){$\mathcal{H}$}
\caption{Hyperbolic Steiner chain}\label{3-hyperbolic}
\end{center}
\end{figure}

\section{Focally Embedded Simplices in Focal Quadrics }\label{sec:embededsimplices}

Assume $\Delta^4$  is a regular $4$-simplex  with vertices $\{p_1, p_2, p_3, p_4, p_5\}$. For every edge $p_ip_j$, let $p_{ij}$ be the midpoint of $p_ip_j$, and for every triangle with vertices $p_i,p_j,p_k$, let $p_{ijk}$ denote its barycenter.

Let $E$ be an ellipse and $\mathcal{H}$ an hyperboloid of revolution in $4$-space, and assume they are focal quadrics with common axes $l$.
We say that \emph{$\Delta^4$ is focally embedded in $(E,\mathcal{H})$} if (after some permutation of the vertices) the following three conditions hold:
\begin{enumerate}[1)]
\item the segment $p_1p_2$ is orthogonal to $l$, passes through the midpoint $p_{12}$, and  $\{p_1, p_2\}\subset E$;
\item the equilateral triangle $p_3p_4p_5$ is orthogonal to $l$ at  $p_{345}$ and $\{p_3, p_4,p_5\}\subset H$; and
\item there is a piece of surface $\mathcal{H}_{345}$ of the hyperboloid of revolution $\mathcal{H}$ whose boundary is the circumcircle of the triangle $p_3p_4p_5$, and a subarc $E_{12}$ of the ellipse $E$ joining $p_1$ with $p_2$ such that 
$\mathcal{H}\cap \Delta^4=\{p_3,p_4, p_5\}$ and  $E_{12}\cap \Delta^4=\{p_1,p_2\}$.
\end{enumerate}
  
In this section we use the notions of focal quadrics and Steiner chains to prove that the envelope of a Steiner chain of $4$-balls  whose centers lie in a pair of focal quadrics gives rise to two ``opposite" $3$-dimensional submanifolds of $\R^4$ in which the property of constant width is achieved.

\subsection{A focally embedded regular $4$-simplex}
\label{subsec:focallyembedded}
 
Let us assume for simplicity  that the  $3$-space has coordinates $(x,y,z)$ and that the ellipse $E$ and the hyperbola $H$ have the $x$-axis as a common axis $l$. The ellipse $E$ is contained in the $\{x,z\}$-plane and the hyperbola  $H$ in the $\{x,y\}$-space. We are going to choose for convenience the following pair of focal conics, whose equations are:
\begin{equation}\label{EHq}
E\colon z^2=\frac{1}{2} \bigg(  1- \frac{2x^2}{3}\bigg) \quad\quad \text{ and } \quad\quad 
H\colon y^2=\frac{1}{2}\bigg(x^2-1\bigg),
\end{equation} 
with foci at $(1,0,0)$ and $(-1,0,0)$ in the case of $E$ (see Figure \ref{Delta5confocally} left) and   
with foci at $(\sqrt{\frac{3}{2}},0,0)$ and $(-\sqrt{\frac{3}{2}},0,0)$ in the case of $H$ (see Figure \ref{Delta5confocally} right).


\begin{lemma}\label{embeding3d}
Given the ellipse $E$ and the hyperbola $H$ as above, there is a set of four points $\Tilde{\Delta}= \{p_1, p_2, p_3, p'_3\}$ in such a way the following conditions hold:
\begin{enumerate}[1)]
\item The segment $p_1p_2$ is orthogonal to $l$ and $\{p_1, p_2\}\subset E$,
\item the segment $p_3p'_3$ is orthogonal to $l$ and $\{p_3, p'_3\}\subset H$,
\item there is a subarc $H_{34}$ of the hyperbola $H$ joining $p_3$ with $p'_3$ and a subarc $E_{12}$ of the ellipse $E$ joining $p_1$ with $p_2$ such that $H_{34}\cap \Tilde{\Delta}=\{p_3,p_4\}$ and $E_{12}\cap \Tilde{\Delta}=\{p_1,p_2\}$, and
\item $p_1p_3=p_1p_2=p_3p_2=p_3'p_1=p'_3p_2$, and \, $\frac{\sqrt{3}}{2}p_3p'_3=p_1p_2.$
\end{enumerate}
\end{lemma}

\begin{proof}


We shall first prove that for every $1<x_0<\sqrt{\frac{21}{17}}$ 
there is $x_1\in \big(x_0, \sqrt{\frac{3}{2}}\big)$ such that if 
 $p_1= (x_1,0,z_1), \quad p_2= (x_1,0,-z_1),\quad
p_3=(x_0, y_0,0),\quad  p'_3=(x_0, -y_0,0),$
where $z_1=\sqrt{\frac{1}{2} \big(  1- \frac{2x_1^2}{3}\big)}$ and 
$y_0=\sqrt{\frac{1}{2} \big(x_0^2-1\big)}$, then 
\begin{gather*} p_1p_3=p_2p_3=p_1p_3'=p_2p_3',\\
\frac{\sqrt{3}}{2}p_3p'_3=p_1p_2,\\
p_1p_3=p_1p_2.\end{gather*}

The first equalities follow from the fact that $(E,H)$ is a pair of focal conics and Corollary 3.4 of \cite{AMO}. To achieve the second, choose $1<x_0<\sqrt{\frac{21}{17}}$. 
Since we need $\frac{\sqrt{3}}{2}p_3p'_3=p_1p_2$,  observe that   

\begin{equation}\label{prop}
\frac{\sqrt{3}}{2}y_0=z_1;
\end{equation}
hence, $3(x_0^2-1)=4\left( 1-\frac{2x_1^2}{3}\right) $. Therefore,
\begin{equation}\label{x1} 
x_1^2=\frac{-9x_0^2+21}{8}.
\end{equation}
\bigskip

On the other hand, if we want that $p_3p_1=p_1p_2$, then
\[\big(x_1-x_0\big)^2=\frac{5}{4}y_0^2=\frac{5}{8}\big(x_0^2-1\big).\]

In particular, 
\begin{equation}\label{x1-x0}
  x_1-x_0=\frac{\sqrt{5}}{2}y_0.   
\end{equation}

Furthermore, $x_1^2+x_0^2=\frac{-x_0^2+21}{8}$, then 
\[ \frac{-6x_0^2 + 26}{8}=2x_0x_1.\]
Then by (\ref{x1}), 
\[ 4x_0^2x_1^2= x_0^2\left( \frac{-9x_0^2+21}{2}\right).\]

Hence $t=x_0^2>1$ is the unique positive solution of the quadratic equation 
\[ \left(   \frac{-6t + 26}{8}       \right)^2=\frac{-9t^2+21t}{2},\]
which yields
\begin{equation}\label{valx2}
x_0^2=\frac{41-4\sqrt{10}}{27},
\end{equation}
and therefore, by \ref{x1}, we have the following solution to our conditions:
\begin{equation}\label{valyx}
y_0^2=\frac{7-2\sqrt{10}}{27} \quad \mbox{ and }\quad x_1^2=\frac{11+2\sqrt{10}}{12}.
\end{equation} 

\end{proof}
\medskip

The next equality, which follows immediately from  (\ref{valyx}), will be of vital importance later: 

\begin{equation}\label{val32}
x_1^2+ \frac{9}{4}y_0^2= \frac{3}{2}.
\end{equation}

As a consequence of Lemma \ref{embeding3d} we observe next that given a pair of focal quadrics, an ellipse $E$, and $\mathcal{H}$, a hyperboloid of revolution in $4$-space  with common axis the $x$-axis, it is always possible  to focally embed a regular $4$-simplex $\Delta^4$ in $(E,\mathcal{H})$.  

Again, for simplicity, we may assume that 
the $x$-axis is used to obtain the hyperboloid of revolution $\mathcal{H}$ such that the pair $(E,\mathcal{H})$
is a pair of focal quadrics, assuming the order of the variables in $\R^4$ is $(x,y,z,w)$ where their equations are given by (see Figure \ref{Delta5confocally}):  

\begin{equation}\label{EH}
E\colon z^2=\frac{1}{2} \bigg(  1- \frac{2x^2}{3}\bigg) \quad\quad \text{ and } \quad\quad 
\mathcal{H}\colon y^2=\frac{1}{2} \bigg(x^2-1\bigg)-w^2.
\end{equation}


\begin{lemma}
Given the ellipse $E$ and the hyperboloid of revolution $\mathcal{H}$ as above, there is a set of five points $\Delta_4=\{p_1, p_2, p_3, p_4, p_5\}$ that are the vertices of a regular 
$4$-simplex focally embedded in $(E, \mathcal{H})$. 
\end{lemma}

\begin{proof}    

Consider the following set of vertices: 
\begin{gather*}
p_1= (x_1,0,z_1,0), \quad p_2= (x_1,0,-z_1,0),\\
p_3=(x_0,y_0, 0,0),\quad  p_4=(x_0, -\frac{1}{2}y_0,0,-\frac{\sqrt{3}}{2}y_0),\quad 
p_5=(x_0, -\frac{1}{2}y_0,0, \frac{\sqrt{3}}{2}y_0),\end{gather*}
where  $z_1=+\sqrt{\frac{1}{2} \big(  1- \frac{2x_1^2}{3}\big)}$, 
$y_0=+\sqrt{\frac{1}{2} \big(x_0^2-1\big)}$, $x_0$ and $x_1$ as in Lemma \ref{embeding3d}. Then it is easy to see that 
$\{p_1,p_2, p_3,p_4,p_5\}$ is a regular simplex of length  $2z_1$. (See Figure \ref{Delta5confocally}). 

\end{proof}

\begin{figure}[h!]
\centering
\begin{subfigure}[b]{0.5\linewidth}
\includegraphics[width=\linewidth]{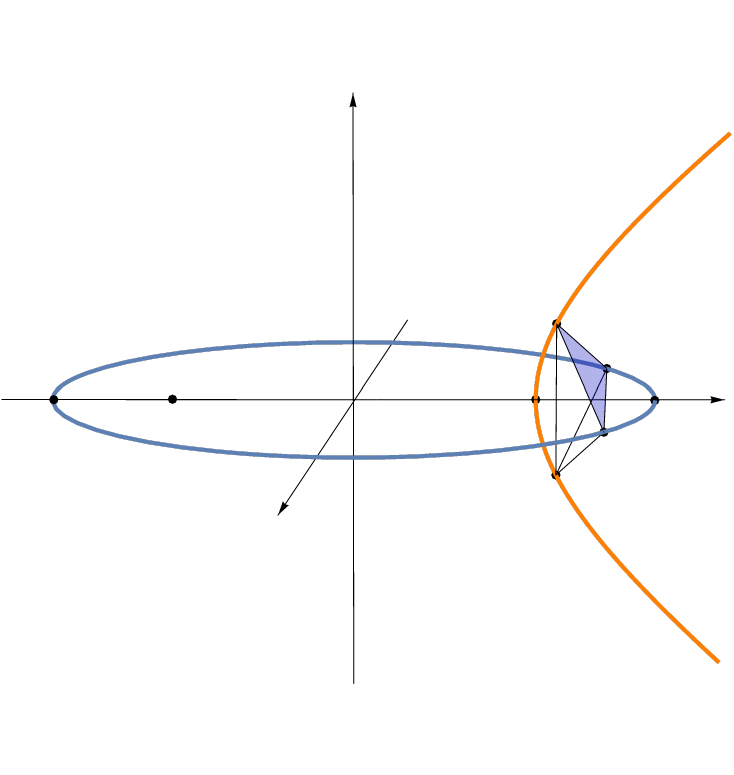}\put(2,107){$x$}\put(-106,196){$y$}\put(-138,68){$z$}\put(-40,90){$p_1$}\put(-38,125){$p_2$}\put(-62,136){$p_3$}\put(-65,82){$p_3'$}\put(-25,95){\small$(\frac{\sqrt{3}}{2},0,0)$}\put(-90,100){\small$(1,0,0)$}\put(-164,100){\small$(-1,0,0)$}
\end{subfigure}
\begin{subfigure}[b]{0.45\linewidth}
\includegraphics[width=\linewidth]{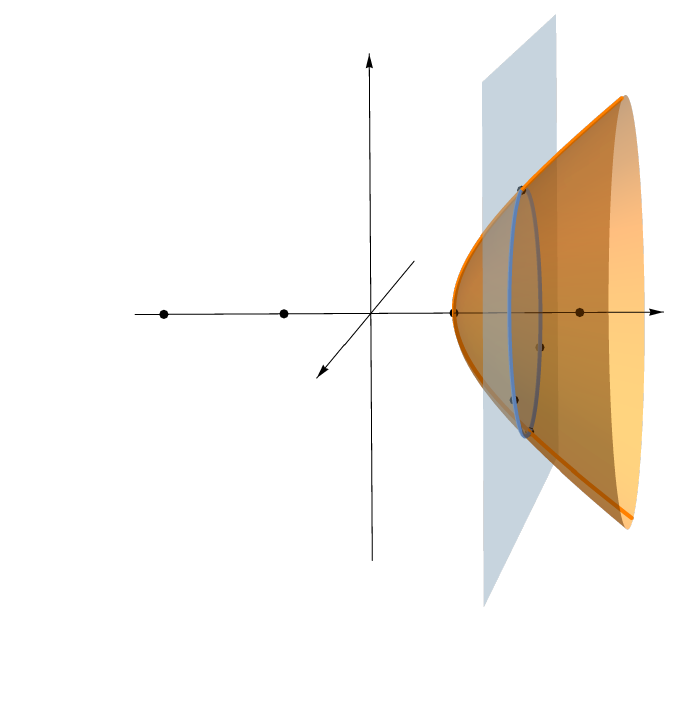}\put(-2,107){$x$}\put(-86,180){$y$}\put(-110,84){$w$}\put(-54,150){$p_3$}\put(-60,80){$p_4$}\put(-38,98){$p_5$}\put(-50,66){$p_3'$}\put(-38,118){\small$(\frac{\sqrt{3}}{2},0,0)$}\put(-90,118){\small$(1,0,0)$}
\end{subfigure}
\caption{A focally embedded pair $(E,\mathcal{H})$}\label{Delta5confocally}
\label{fig:westminster}
\end{figure}
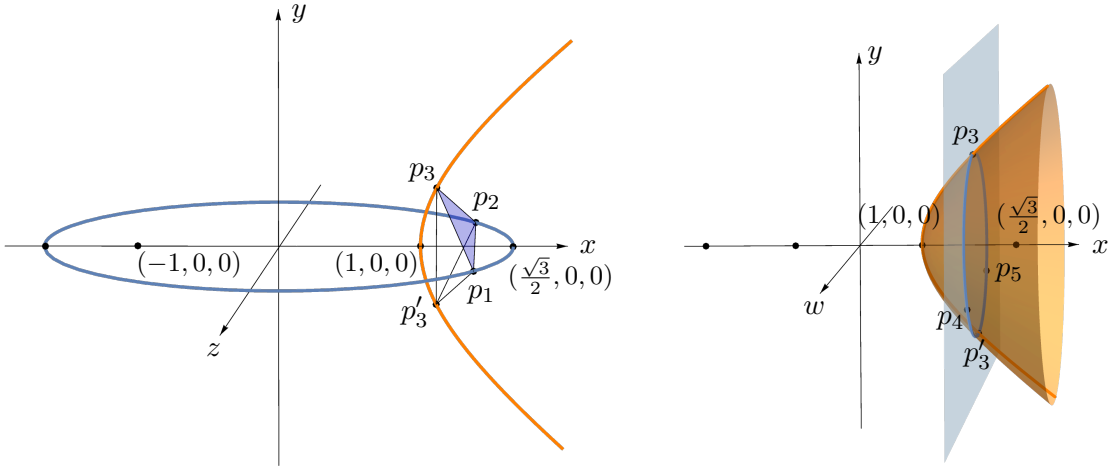

\begin{remark}\label{nota}
We have been assuming for simplicity that the equations of the ellipse and the hyperbola are parametrized as in Equations \eqref{EHq}. For the general case, up to congruence, any pair of focal quadrics may be parametrized as in Equation \eqref{ellipse-hyperbola}, and up to homothety, 
 by any real number $1<a<\infty$, as follows:
\begin{equation}\label{EHa}
E\colon z^2=(a^2-1) \big(  1- \frac{x^2}{a^2}\big) \quad\quad \text{ and } \quad\quad 
H\colon y^2=(a^2-1)\big(x^2-1\big),
\end{equation} 
where $1=c^2=b^2-a^2$.


Then, as in the above section, we can prove that 
for every $1<a<\infty$, there are unique numbers $1<x_0<x_1<a$ such that if 
\begin{gather*} p_1= (x_1,0,0,z_1), \quad p_2= (x_1,0,0,-z_1),\\
p_3=(x_0,y_0, 0,0),\quad  p_4=(x_0, -\frac{1}{2}y_0,0, -\frac{\sqrt{3}}{2}y_0),\quad 
p_5=(x_0, -\frac{1}{2}y_0,0, \frac{\sqrt{3}}{2}y_0),\end{gather*}
where $z_1=+\sqrt{(a^2-1) \big(  1- \frac{x_1^2}{a^2}\big)}$ and 
$y_0=+\sqrt{(a^2-1) \big(x_0^2-1\big)}$.  Then the simplex $\Delta^4$ with vertices 
$\{p_1,p_2,p_3,p_4,p_5\}$ is regular of length $2z_1$ and is focally embedded in $(E,\mathcal{H})$, where
\begin{equation}\label{EEHa}
E\colon z^2=(a^2-1) \bigg(  1- \frac{x^2}{a^2}\bigg) \quad \text{ and } \quad
\mathcal{H}\colon y^2=(a^2-1)\big(x^2-1\big)-w^2.
\end{equation}


Summarizing,  we may assume throughout the rest of the paper the general framework that $E$ and $\mathcal{H}$ are a pair of focal quadrics with equations as in Equation \eqref{EH}. Thus, $E$  is an ellipse in the $\{x,z\}$-plane with foci at $(1,0,0,0)$ and $(-1,0,0,0)$ and $\mathcal{H}$ is an hyperboloid of revolution in the   $\{x,y,w\}$-space with foci at $f_h^+=(\sqrt{3/2},0,0,0)$ and $f_h^-=(-\sqrt{3/2},0,0,0)$. Furthermore, the simplex $\Delta^4$ with vertices 
$\{p_1,p_2,p_3,p_4,p_5\}$ is regular of length $2z_1$ and is focally embedded in $(E,\mathcal{H})$, 
where
\begin{gather*} p_1= (x_1,0,z_1,0), \quad p_2= (x_1,0,-z_1,0),\\
p_3=(x_0,y_0, 0,0),\quad  p_4=(x_0, -\frac{1}{2}y_0,0, -\frac{\sqrt{3}}{2}y_0),\quad 
p_5=(x_0, -\frac{1}{2}y_0,0, \frac{\sqrt{3}}{2}y_0),\end{gather*}
and where $z_1=+\sqrt{(a^2-1) \big(  1- \frac{x_1^2}{a^2}\big)}$ and 
$y_0=+\sqrt{(a^2-1) \big(x_0^2-1\big)}$.

\end{remark}

\subsection{Achieving constant width in a pair \boldmath{$\Omega_E$} and \boldmath{$\Omega_{\mathcal{H}}$}.}\label{sec:distance}

The purpose of this section is to use the notions of focal quadrics and Steiner chains to  construct a family of $4$-balls centered at the faces of the focal pair $(E,\mathcal{H})$ whose envelope is  a pair of $3$-dimensional submanifolds in $\R^4$. This pair of submanifolds are in a certain sense ``opposite," since when they are both present together, the property of constant width is achieved.

Consider the plane $\Gamma_{345}$ containing $\{p_3,p_4,p_5\}$ in the $\{x,y,w\}$-space intersecting $\mathcal{H}$ in the circle $C$ with center on the $x$-axis. Let $\mathcal{H}_{345}$ be the bounded surface of $\mathcal{H}$ with boundary $C$.  
 Let $\mathbb{S}^+$ be the $2$-dimensional sphere with center at  $f_h=(\sqrt{3/2},0,0)$ containing $C$ and let  $\mathbb{S}^-$ be the $2$-dimensional sphere with center at  
 $f_h^-=(-\sqrt{3/2},0,0)$ containing $C$. As in Section \ref{hiperbolicStainer}, consider the collection $\Omega_ {\mathcal{H}}$ of $3$-balls in the $\{x,y,w\}$-space  contained between $\mathbb{S}^+$ and 
$ \mathbb{S}^-$ and tangent to both spheres.  This collection of $3$-balls will be a hyperbolic Steiner chain $\Omega_{\mathcal{H}}$ centered at $\mathcal{H}$ and based in $C$.  In particular, note that centers of all these balls lie on $\mathcal{H}_{345}$.

Similarly, we now work  in the $\{x,z\}$-plane. As before, denote by $S^+$ the circle with center at  $f_e=(1,0,0,0)$ containing $\{p_1,p_2\}$ and by $S^-$ the circle with center at  $f_e^-=(-1,0,0,0)$ containing $\{p_1,p_2\}$. 
Consider the elliptic Steiner chain $\Omega_E$ of disks as in section \ref{Elliptic stainer} contained in $S^+\setminus S^-$ and tangent to both circles. We know that the centers of all these disks lie on the subarc of $E$ with endpoints at $\{p_1, p_2\}$. Denote by $E_{12}$ this subarc of the ellipse $E$.
Finally, for every $\boldsymbol{x} \in \mathcal{H}_{345}$, denote by $\mathcal{B}_{\boldsymbol{x}}$ the $4$-ball  in  $\mathbb{R}^4$ centered at $\boldsymbol{x}$ such that the intersection of $\mathcal{B}_{\boldsymbol{x}}$ with the $\{x,y,w\}$-plane is an element of $\Omega_{\mathcal{H}}$, and let $\mathcal{R}_{\boldsymbol{x}}$ be the radius of this ball.  Similarly,  for every ${\boldsymbol{y}}\in E_{12}$, denote by $B_{\boldsymbol{y}}$ the $4$-ball  in  $\mathbb{R}^4$  centered at ${\boldsymbol{y}}$ that is an element of $\Omega_E$, and let $R_{\boldsymbol{y}}$ be the radius of this ball.

\medskip
Remember that the length of every edge of the regular $4$-simplex $\Delta^4=\{p_1,p_2,p_3,p_4, p_5\}$ focally  embedded in $(E, \mathcal{H})$ is $2z_1$  (according to the numbers obtained in Equations \eqref{valyx} and by Equation \eqref{prop},  we know that  $2z_1=\frac{\sqrt{7-2\sqrt{10}}}{3}$. Then the following statement is true:

\begin{lemma}\label{constante}
For every ${\boldsymbol{x}}\in \mathcal{H}_{345}$ and for every  ${\boldsymbol{y}}\in E_{12}$, 
\[ {\boldsymbol{x}}{\boldsymbol{y}}+\mathcal{R}_{\boldsymbol{x}}+R_{\boldsymbol{y}} = 2z_1.\]

\end{lemma}

\begin{proof}
By construction, we observe first that: 
\begin{enumerate}[1)] 
\item $\mathcal{R}_{\boldsymbol{x}}+ {\boldsymbol{x}}f_h=\mbox{ radius of } \mathbb{S}^+$, where $f_h=(\sqrt{3/2},0,0,0)$  is the focus of the hyperboloid $\mathcal{H}$, and
\item $R_{\boldsymbol{y}} + {\boldsymbol{y}}f_e= \mbox{radius of } S^+$, where  $f_e=(1,0,0,0)$  is the focus of the ellipse $E$.
\end{enumerate}

Consequently 
\[ {\boldsymbol{x}}{\boldsymbol{y}}+\mathcal{R}_{\boldsymbol{x}}+R_{\boldsymbol{y}}= {\boldsymbol{x}}{\boldsymbol{y}}- {\boldsymbol{x}}f_h -{\boldsymbol{y}}f_e + \mbox{ radius of } \mathbb{S}^+  +  \mbox{ radius of } S^+\] 
which is constant by Theorem \ref{Runoyuno}.  Furthermore if ${\boldsymbol{x}}=p_3$ and ${\boldsymbol{y}}=p_1$, 
then $\mathcal{R}_{\boldsymbol{x}}=0=R_{\boldsymbol{y}}$ and hence 
$ {\boldsymbol{x}}{\boldsymbol{y}}+\mathcal{R}_{\boldsymbol{x}}+R_{\boldsymbol{y}} = p_3p_1=2z_1$.
\end{proof}


Let us consider the two Steiner chains $\Omega_{\mathcal{H}}$ and $\Omega_E$ defined as before, and let 
\[|\Omega_{\mathcal{H}}|=\bigcup_{x\in \mathcal{H}_{345}} \mathcal{B}_{\boldsymbol{x}} \quad \mbox{ and } \quad |\Omega_E|=\bigcup_{y\in E_{12}}B_{\boldsymbol{y}} .\]

 Recall that for two subsets $S, T$ of $\mathbb{R}^4$, the \emph{diameter} between $S$ and $T$ is defined by $d(S,T)= \sup\{ ab\mid a\in S, b\in T\}.$
Then as a consequence of Lemma \ref{constante}, we are able to show that the diameter between $|\Omega_{\mathcal{H}}|$ and $|\Omega_E|$ is $2z_1$.

\begin{lemma}\label{lemS} 
If $\Omega_{\mathcal{H}}$ and $\Omega_E$ are the two Steiner chains defined above, then  
\[d(\left|\Omega_{\mathcal{H}}\right|, \left|\Omega_E\right|)\leq2z_1.\]
\end{lemma}

\begin{proof}
 Let $w\in |\Omega_{\mathcal{H}}|$.  Then $w\in \mathcal{B}_{\boldsymbol{x}},$  
for some point ${\boldsymbol{x}}\in \mathcal{H}$. Similarly, let $u\in |\Omega_E|$.  Then $u\in B_{\boldsymbol{y}},$  
for some point ${\boldsymbol{y}}\in E_{12}$. Therefore, by the triangle inequality, 
$wu\leq \x\y+ \mathcal{R}_{\x}+R_{\y}$. Then
by Lemma \ref{constante},  $wu\leq 2z_1$.   
\end{proof} 


Next, we show that the envelope of $\Omega_{\mathcal{H}}$ and $\Omega_E$ is a pair of $3$-dimensional submanifolds in $\R^4$, in which the property of constant width is achieved.

\smallskip

For every $\x\in {\mathcal{H}}_{345}$  and $\y\in E_{12}$, let  $L(\x,\y)$ 
be the line through $\x$ and $\y$. (See  Figure \ref{isaac}.)
Denote by $\phi_1({\x},{\y})$ the point in $L({\x},{\y})$ at  distance $\mathcal{R}_{\x}$ from ${\x}$, where ${\x}$ is between $\phi_1({\x},{\y})$ and ${\y}$, and let $\phi_2({\x},{\y})$ be the point in $L({\x},{\y})$ at  distance $R_{\y}$ from ${\y}$, where ${\y}$ is between ${\x}$ and $\phi_2({\x},{\y})$.  Note that $\phi_1({\x},{\y})\in |\Omega_{\mathcal{H}}|$, $\phi_2({\x},{\y})\in |\Omega_E|$ and   
$\phi_1({\x},{\y})\phi_2({\x},{\y})= 2z_1$. Therefore, by Lemma \ref{constante}, $d(|\Omega_{\mathcal{H}}|, |\Omega_{E}|)=2z_1$.

\begin{figure}[h]
\begin{center}
\includegraphics[scale=.35]{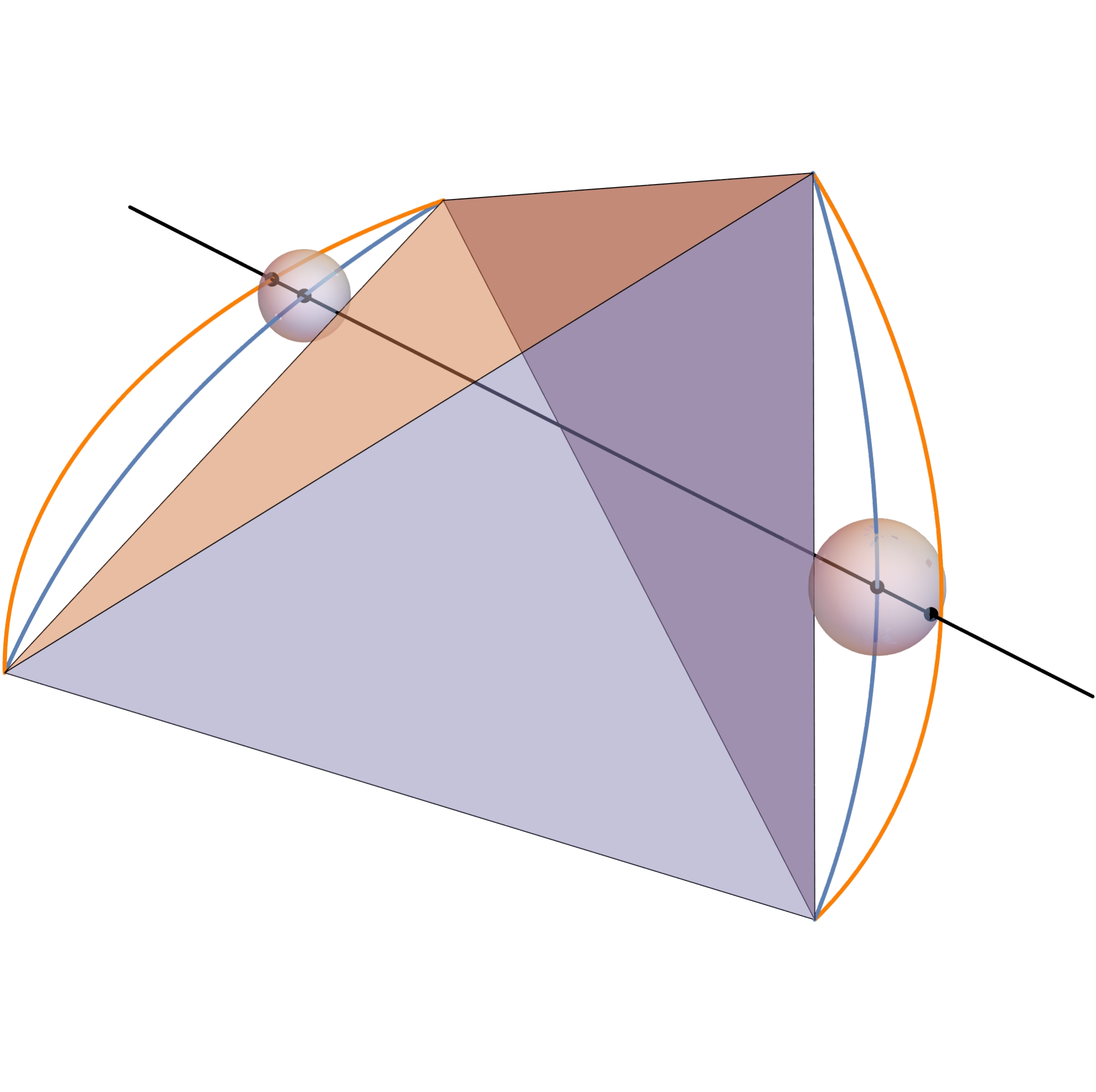}\put(2,63){$L({\x},{\y})$}\put(-25,82){$\phi_2({\x},{\y})$}\put(-35,88){${\y}$}\put(-135,137){${\x}$}\put(-176,132){$\phi_1({\x},{\y})$}\put(-190,63){$p_4p_5$}\put(-52,18){$p_2$}\put(-50,158){$p_1$}\put(-111,154){$p_3$}
\caption{The line $L({\x},{\y})$ }\label{isaac}
\end{center}
\end{figure}

\begin{lemma}
For every ${\x}\in {\mathcal{H}_{345}}$ and ${\y}\in E_{12}$, 
the interval $\phi_1({\x},{\y})\phi_2({\x},{\y})$ is a binormal of the convex hull  $\conv(|\Omega_{\mathcal{H}}|\cup|\Omega_E|)$.  
Furthermore, any two of these binormals are non-parallel.
\end{lemma}
\begin{proof} If ${\x}\notin C$, then there is a unique support plane of $\conv(|\Omega_{\mathcal{H}}|\cup|\Omega_E|)$ at $\phi_1({\x},{\y})$ orthogonal to $\phi_1({\x},{\y})\phi_2({\x},{\y})$ because this is the unique tangent plane to $\mathcal{B}_{\x}$ at $\phi_1({\x},{\y})$. Similarly, if ${\y}\not=\{p_1,p_2\}$, then 
there is a unique support plane of $\conv(|\Omega_{\mathcal{H}}|\cup|\Omega_E|)$ at $\phi_2({\x},{\y})$ orthogonal to $\phi_1({\x},{\y})\phi_2({\x},{\y})$, because this is the unique tangent plane to $B_{\y}$ at $\phi_2({\x},{\y}).$ 
Note that if ${\x}\in C$, then $\phi_1({\x},{\y})={\x}$  for every ${\y}\in E_{12}$,  
and similarly, $\phi_2({\x},p_i)=p_i$  for every 
${\x}\in \mathcal{H}_{345}$, and $i=1,2$.  In any other case, then
\begin{gather*}\phi_1\colon (\mathcal{H}_{345} \setminus C)\times E_{12} \to \mathbb{R}^4\text{ is an smooth embedding, and } \\ 
\phi_2\colon \mathcal{H}_{345}\times (E_{12}\setminus \{p_1,p_2\}) \to \mathbb{R}^4\text{ is an smooth embedding, as well.}\end{gather*}
Therefore the interval $\phi_1({\x},{\y})\phi_2({\x},{\y})$ is a binormal of the convex hull  $\conv(|\Omega_{\mathcal{H}}|\cup|\Omega_E|)$ and of course, all these binormals are non-parallel.
\end{proof}   


We may regard 
\[\phi_1 (\mathcal{H}_{345} \times E_{12}) \quad \mbox{ and }\quad  \phi_2( \mathcal{H}_{345}\times E_{12}),\]
 the envelopes of $|\Omega_{\mathcal{H}}|$ and $|\Omega_E|$, 
as two $3$-dimensional submanifolds of $\R^4$ in which the property of constant width is achieved.

\section{A Focal 2-Skeleton of the Regular 4-Simplex }\label{sec:2-skeleton}

As we mentioned before, our $4$-dimensional body of constant width will be constructed from the $4$-simplex Reuleaux body by replacing a small neighborhood of its $2$-skeleton with sections of an envelope of Steiner chains of $4$-balls.  The centers of these balls lie in the $2$-skeleton 
of the $4$-simplex where every edge is a subarc of an ellipse  and every triangle is a piece of surface of a hyperboloid of revolution  in such a way that  every pair of dual faces (edge--triangle) are focal.  

In the same way that we constructed the focal pair formed by the subarc of ellipse $E_{12}$ and the hyperboloid of revolution $\mathcal{H}_{345}$, in this section we will explicitly construct the $2$-skeleton referred above. The idea now is to construct two new focal pairs formed say, by the subarc of ellipse $E_{45}$ and the hyperboloid of revolution $\mathcal{H}_{123}$, where $E_{45}$ lies in the hyperboloid of revolution  $\mathcal{H}_{345}$ and at the same time will be part of boundary of $\mathcal{H}_{345}$.
 
We begin by noting that both $p_{123}$, the barycenter of the triangle  $p_1p_2p_3$ and $p_{45}$, the midpoint of the segment  $p_4p_5$, lie in the $\{x,y\}$-plane because $p_{45}=\big(x_0, -\frac{y_0}{2},0,0\big)$ and 
$p_{123}=\big(\frac{2x_1+x_0}{3}, \frac{y_0}{3},0,0\big)$.  
Since by (\ref{x1-x0}),  $p_{123}-p_{45}=\big(\frac{2(x_0-x_1)}{3}, \frac{5y_0}{6},0,0\big)=
\big(\frac{\sqrt{5}}{3}y_0, \frac{5y_0}{6},0,0\big)$, the line $L$ through $p_{123}$ and $p_{45}$ has unit vector $u_L$ given by 
\begin{equation}\label{pendiente}
u_L=\frac{p_{123}-p_{45}}{\|p_{123}-p_{45}\|}=\bigg(\frac{2}{3}, \frac{\sqrt{5}}{3},0,0\bigg).
\end{equation}

Then the following lemma is true: 
      
\begin{lemma}\label{cachito}
Let $\omega\in \mathcal{H} \cap L$ be the point that lies in $L$ in such a way that $p_{45}$ is between $\omega$ and  the barycenter $p_{123}$.  Then the distance between $\omega$ and $p_{45}$ is equal to $\sqrt{\frac{3}{2}}-x_1$. 
 \end{lemma} 
 
 \begin{proof}  Since both the midpoint $p_{45}$ of the segment $p_4p_5$ and the barycenter $p_{123}$ of the triangle $p_1p_2p_3$ lie in the  $\{x,y\}$ plane, during this proof we shall work in the $\{x,y\}$-plane. Therefore, the hyperboloid of revolution is actually $H$ with equations as in Equations \eqref{EHq}. 
 By  Equation \eqref{pendiente}, the unit vector parallel to $L$ is given by $\big(\frac{2}{3}, \frac{\sqrt{5}}{3}\big)$, so \[L=\bigg\{\bigg(x_0,-\frac{y_0}{2}\bigg)-t\bigg(\frac{2}{3}, \frac{\sqrt{5}}{3}\bigg)\bigg | \ t\in \R\bigg\}.\]
 Therefore, the line $L$ consists of the points 
$\big(x_0-\frac{2}{3}t, -(\frac{y_0}{2}+ \frac{\sqrt{5}}{3})t\big),$
with $t\in \R$.  We wish to prove that if $t=\sqrt{3/2}-x_1$, then  
\[\big(x_0-\frac{2}{3}t,  -(\frac{y_0}{2}+ \frac{\sqrt{5}}{3}t) \big)\in H.\]

Since the equation of the hyperbola $H$ is $2y^2=x^2-1$, in order to find the intersection of $H$ with $L$ we need to solve the following equation:
\[\big(x_0-\frac{2}{3}t\big)^2-1= 2\big(\frac{y_0}{2}+\frac{\sqrt{5}}{3}t\big)^2.\]

Hence 
\[x_0^2-\frac{4x_0}{3}t+\frac{4}{9}t^2-1=\frac{y_0^2}{2}+\frac{2y_0\sqrt{5}}{3}t+\frac{10}{9}t^2.\]
Consequently
\[0=\frac{6}{9}t^2+ \frac{2(y_0\sqrt{5} +2x_0 )}{3}t+\frac{y_0^2}{2}-2y_0^2.\]

Recall that by Equation \eqref{x1-x0}, $x_1-x_0=\frac{\sqrt{5}}{2}y_0$, then $2x_1=y_0\sqrt{5}+2x_0$, 
hence we obtain the equation
\[t^2+2x_1 t-\frac{9}{4}y_0^2=0,\]
whose positive solution is:
\[-x_1+\sqrt{x_1^2+\frac{9}{4}y_0^2}.\]

Finally, recall from Equation \eqref{val32}) that $x_1^2+\frac{9}{4}y_0=3/2$ because
\[x_1^2= \frac{    11+2\sqrt{10}             }{12}\quad \mbox{ and } 
\quad  y_0^2= \frac{    7-2\sqrt{10}             }{27}.\]

Consequently,  if $t=\sqrt{3/2}-x_1$, then 
\[\big(x_0,  -\frac{y_0}{2})-t( \frac{2}{3}, \frac{\sqrt{5}}{3})\in H\subset \mathcal{H}\]
as we wished. 
\end{proof}

\smallskip 
Denote by $\Gamma_{45}$ the plane generated by $L$ and the line through $p_4p_5$.  Consider the quadric 
$E'=\Gamma_{45}\cap \mathcal{H}$. Let  $v_{E'}$ be the tangent vector of the conic $E'$ at $p_4$ and let $v_E$ be the tangent vector of the ellipse $E$ at the point $p_1$ (see Figure \ref{tangentrotation}).  If $u$ and $v$ are two vectors in $\R^4$, let us denote by $\measuredangle\big(u,v\big)$ the angle between them. Then the following  is true: 
  
\begin{lemma} \label{angulos}
$\tan\big(\measuredangle\big(v_E, (0,0,1,0)\big)\big)=\tan\big(\measuredangle\big(v_{E'}, (0,0,0,1)\big)\big)$.
\end{lemma}
\begin{proof} 
We observe first that the  tangent vector of the ellipse $E$ at the point $p_1$ can be obtained by taking the derivative of ellipse equation \eqref{EH}, which is 
\[v_E=\bigg(-1,0,\frac{x_1}{3z_1},0\bigg),\] 
and the tangent of the angle between $v_E$ and $(0,0,1,0)$ is 
\begin{equation}
\tan\big(\measuredangle\big(v_E, (0,0,1,0)\big)=-\frac{3z_1}{x_1}.
\end{equation}

Similarly, to calculate the tangent vector $v_{E'}$ of the quadric $E'=\Gamma\cap\mathcal{H}$ at $p_4$,
consider a normal vector $\eta_{\mathcal{H}}$ of $\mathcal{H}$ at $p_3$ in the $\{x,y,w\}$-space, by taking the derivative of the hyperbola equation \eqref{EH} to obtain that 
\[\eta_{\mathcal{H}}=(-x_0, 2y_0,0),\]
which is normal to the surface of revolution $\mathcal{H}$ at $p_3$ in the $\{x,y,w\}$-space.  

Next, let us rotate $\eta_{\mathcal{H}}$ by an angle of $120^{\circ}$ around the $x$-axis in the $\{x,y,w\}$-space with the matrix   
\[
\begin{pmatrix}
1 & 0 & 0 \\
0, &-1/2 & \sqrt{3}/2 \\
0,& - \sqrt{3}/2 &-1/2
\end{pmatrix}\]
and obtain  
\[\eta'_{\mathcal{H}}=(-x_0, -y_0, -\sqrt{3}y_0),\]
which is normal to the surface of revolution $\mathcal{H}$ at $p_4$ in the $\{x,y,w\}$-space.  

On the other hand, by Equation \eqref{pendiente}, the normal vector to the plane $\Gamma$ at the $\{x,y,w\}$-space is  
\[\eta_\Gamma=(-\sqrt{5},2,0),\]
because, by \eqref{pendiente},  $(2,\sqrt{5})$ is a vector parallel to $L$ in the $\{x,y\}$-space. 

Therefore, the vector $\eta'_{\mathcal{H}}\times \eta_{\Gamma}$ is orthogonal to  $\eta'_{\mathcal{H}}$ and  to $\eta_\Gamma$, hence it 
lies at the tangent plane of the hyperboloid of revolution $\mathcal{H}$ at $p_4$ and also lies  in $\Gamma$; that is,
$\eta'_{\mathcal{H}}\times \eta_{\Gamma}$
is a tangent vector to the quadric $E'=\Gamma\cap \mathcal{H}$ at $p_4$.

A simple calculation 
\begin{equation}
\eta'_{\mathcal{H}}\times \eta_\Gamma=
\begin{pmatrix}
x & y & w \\
-x_0, &-y_0& -\sqrt{3}y_0 \\
-\sqrt{5},& 2 & 0
\end{pmatrix}
\end{equation}
yields
$$\eta'_{\mathcal{H}}\times \eta_\Gamma= \big(2\sqrt{3}y_0, 
\sqrt{5}\sqrt{3}y_0, -(2x_0+\sqrt{5}y_0\big).$$

By Equation \eqref{pendiente}, the unit vector $(2/3,\sqrt{5}/3,0)$ in the $\{x,y,w\}$-space is parallel to $u_L$ and by (\ref{x1-x0}), $2x_0+\sqrt{5}=2x_1$. Then we have that 
 \[\eta'_{\mathcal{H}}\times \eta_\Gamma= \big(3\sqrt{3}y_0\big) (2/3,\sqrt{5}/3,0)  - 2x_1(0,0,1),\]
which implies that the tangent of the angle between $\eta'_{\mathcal{H}}\times \eta_\Gamma$ and the vector $(0,0,1)$ 
is \[\frac{3\sqrt{3}y_0}{-2x_1}.\] Therefore, by Equation \eqref{prop}, we have $-\frac{3\sqrt{3}y_0}{2x_1}=-\frac{3z_1}{x_1}$ and since $\eta'_{\mathcal{H}}\times \eta_\Gamma$ is parallel to $v_{E'}$, we can conclude
\[\tan\big(\measuredangle\big(v_{E'}, (0,0,0,1)\big)\big)=-\frac{3z_1}{x_1}=\tan\big(\measuredangle\big(v_{E}, (0,0,1,0)\big)\big).\]
\end{proof}

\begin{figure}[H]
\begin{center}
\includegraphics[scale=.85,angle=90]{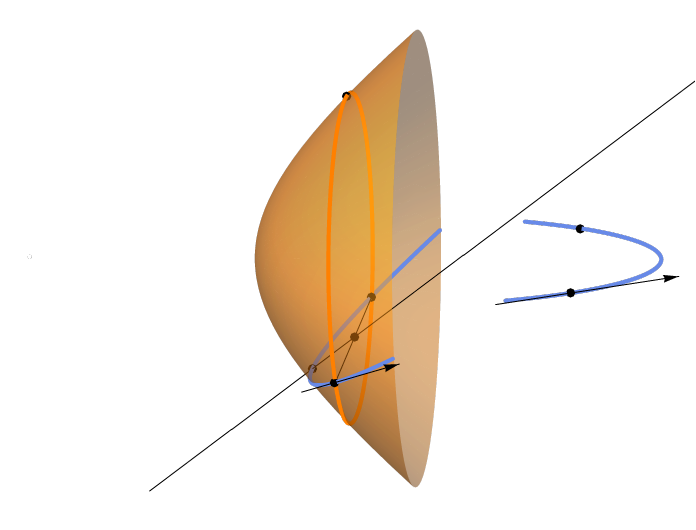}\put(-22,122){\textcolor{blue}{$E$}}\put(-102,122){\textcolor{blue}{$E'$}}\put(-104,190){\textcolor{orange}{$\mathcal H$}}\put(-24,178){$L$}\put(-125,55){$v_{E'}$}\put(-10,92){$v_{E}$}\put(-158,48){$p_4$}\put(-140,102){$p_5$}\put(-60,87){$p_1$}\put(-56,131){$p_2$}\put(-150,186){$p_3$}\put(-163,70){$\omega$}
\caption{The rotation of the ellipse $E$ producing the conic $E'$ over the hyperboloid of revolution $\mathcal{H}$}\label{tangentrotation}
\end{center}
\end{figure}
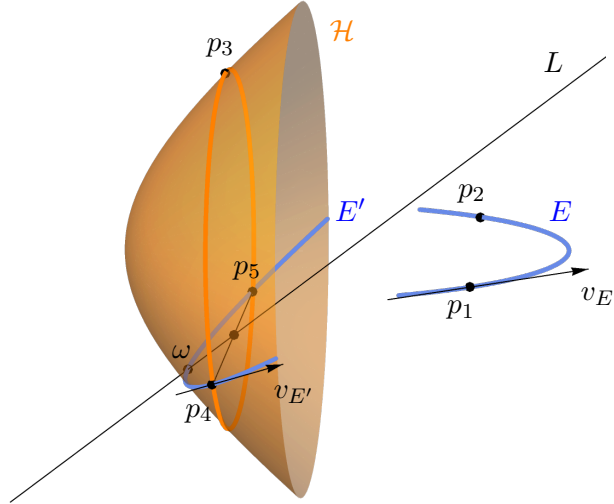

\begin{lemma}\label{milagro}
Let $\Phi\colon \R^4\to \R^4$ be the  linear isometry that fixes the point $p_3$  and is such that  
$\Phi(p_1)=p_4$  and   $\Phi(p_2)=p_5$. Then 
\[\Phi(E)=E_{45}=\mathcal{H}\cap \Gamma_{45}.\]
\end{lemma}
\begin{proof} Both $\Phi(E)$ and $E_{45}$ are conics in $\Gamma_{45}$  containing the points $p_4$ and $p_5$. Furthermore, by Lemma \ref{cachito}, 
$\Phi(\sqrt{3/2})=\omega$; hence $\omega \in \Phi(E) \cap E_{45}$. Also, by  Lemma \ref{angulos}, the tangents of $\Phi(E)$ and $E_{45}$ at $p_4$ coincide.
Similarly, the tangents of $\Phi(E)$ and $E_{45}$ at $p_5$ coincide. Since two plane quadrics which share $5$ points coincide, we have that $\Phi(E)=E_{45}=\mathcal{H}\cap \Gamma_{45}.$ 
\end{proof}


Denote by $0$ the barycenter of the  regular $4$-simplex $\Delta^4= \{p_1,p_2, p_3,p_4,p_5\}$.   For every edge $p_ip_j$ of $\Delta^4$, let $p_{ij}$ be the midpoint of $p_ip_j$, $L_{ij}$, the line through $0$ and $p_{ij}$ and $\Gamma^2_{ij}$, the plane through $0$  and $p_ip_j$.   For every triangle $p_ip_jp_k$ of $\Delta^4$, let $p_{ijk}$ be the barycenter of the triangle $p_ip_jp_k$, $L_{ijk}$, the line through $0$ and $p_{ijk}$ and $\Gamma^3_{ijk}$, the $3$-space through the $0$ and the triangle $p_ip_jp_k$.

\smallskip

For every edge $p_ip_j$ of $\Delta^4$ let $E_{ij}\subset \Gamma^2_{ij}$ be a subarc of an ellipse joining the points $p_i$ and $p_j$, where this ellipse has $L_{ij}$ as axis. For every triangle $p_kp_lp_m$ of $\Delta^4$, let $T_{lkm}\subset \Gamma^3_{lkm}$ be a piece of hyperboloid of revolution whose axis is $L_{lkm}$, satisfying the following four conditions: 
\begin{enumerate}[1)]
\item The boundary of  $T_{lkm}$ is  $E_{lk}\cup E_{lm}\cup  E_{ml}$;
\item the hyperboloid of revolution $T_{lkm}$ and the ellipse $E_{ij}$ are focal quadrics; and
\item the $4$-simplex $\{p_1,p_2, p_3,p_4,p_5\}$ is focally embedded in the pair $\big(T_{ijk}, E_{\lambda\alpha}\big)$, 
where $\{i,j,k,\lambda,\alpha\}=\{1,2,3,4,5\}$.
\item For $\{i,j,k,\lambda,\alpha\}=\{1,2,3,4,5\}$, $\biggl(\bigcup_{i,j,k,\lambda,\alpha}\bigl(E_{ij}\bigcup T_{ijk} \bigr)\biggr)$ has the same group of symmetries as the $4$-simplex $\Delta^4$.
\end{enumerate} 

If for every triple $i,j,k$ there are subarcs of ellipses $E_{\lambda\alpha}$ and pieces of hyperboloids of revolution $T_{ijk}$  satisfying 1), 2), 3) and 4), where $\{i,j,k,\lambda,\alpha\}=\{1,2,3,4,5\}$, then we say that \[\biggl(\bigcup_{i,j,k,\lambda,\alpha}\bigl(E_{ij}\bigcup T_{ijk} \bigr)\biggr)\] %
is a \emph{focal $2$-skeleton of the regular $4$-simplex} $\Delta^4=\{p_1,p_2, p_3,p_4,p_5\}.$

\smallskip

We observe next that Lemma \ref{milagro} allows us to generalize the above construction for every pair $(E_{\lambda \alpha},\mathcal{H}_{ijk})$ and then construct the focal $2$-skeleton of a regular $4$-simplex crucial for the construction of a $4$-dimensional convex body of constant width. That is, we shall define, for every triple $\{i,j,k\}\subset \{1,2,3,4,5\}$, 
where $\{i,j,k,\lambda,\alpha\}=\{1,2,3,4,5\}$, subarcs of ellipses $E_{\lambda\alpha}$ and pieces of hyperboloids of revolution $\mathcal{H}_{ijk}$  satisfying 1), 2), 3) and 4).
\smallskip

Let $\mathcal{H}_{345}$ be the connected piece of the hyperboloid of revolution $\mathcal{H}$ bounded by $\Gamma_{34}\cap\mathcal{H}$, 
$\Gamma_{45}\cap\mathcal{H}$ and $\Gamma_{53}\cap\mathcal{H}$ and note that $\mathcal{H}_{345}$ is invariant under a linear isometry of 
$\Delta^4=\{p_1,p_2, p_3,p_4,p_5\}$ that send $\{p_1,p_2, p_3\}$ onto itself.

For every triple $\{i,j,k\}\subset \{1,\dots,5\}$, let $\Psi$ be any linear isometry of $\Delta^4$ that sends 
$\{p_3,p_4,p_5\}$ to $\{p_i,p_j,p_k\}$ and define $\mathcal{H}_{ijk} =\Psi(\mathcal{H}_{345})$. Note that since $\mathcal{H}_{345}$ is invariant under  linear isometries of 
$\Delta^4$ that send $\{p_1,p_2, p_3\}$ onto itself, then our definition of $\mathcal{H}_{ijk}$ does not depend on our choice of $\Psi$.

Likewise, if $E_{12}$ is the arc of ellipse $E$ with endpoints $\{p_1,p_2\}$ and for every $\{i,j\}\subset \{1,\dots,5\}$, we define 
the subarc of ellipse $E_{ij}$ as $\Psi(E_{12})$, for any  linear isometry of $\Delta^4$ that sends 
$\{p_1,p_2\}$ to $\{p_i,p_j\}$.  Note that our definition of $E_{ij}$ does not depend on our choice of $\Psi$. 

By construction, $\big(E_{\lambda\alpha},\mathcal{H}_{ijk}\big)$ is a pair of focal quadrics, and by Lemma \ref{milagro}, the boundary of
$\mathcal{H}_{ijk}$ is $E_{ij}\cup E_{jk}\cup E_{ki}$.

Define the focal $2$-skeleton of $\Delta^4$ as:
\[ \biggl(\bigcup_{i,j,k,\lambda,\alpha}\bigl(E_{ij}\bigcup \mathcal{H}_{ijk} \bigr)\biggr)\]
and note that it has the symmetries of the regular $4$-simplex $\Delta^4$.

\begin{remark}
 a) The construction of the focal $2$-skeleton in $\Delta^4$, only works for the parameter $a^2=3/2$ (see Remark \ref{nota}).\\
\noindent b) As in the case of dimension $3$ (see \cite{AMO}), it is well known that instead of a focal pair consisting of an ellipse and an hyperbola, it is possible to consider a focal pair consisting of two parabolas. However, unfortunately after extensive and tedious calculations it can be observed that the construction of a focal $2$-skeleton is not possible in this case. 
\end{remark}

\section{The Wedge-Submanifolds and the Assembly}\label{sec:assembly}

As we show in Section \ref{sec:2-skeleton}, the envelope of the $4$-balls of every pair of dual faces of the focal $2$-skeleton gives rise to two ``opposite" $3$-dimensional submanifolds in which we achieve the property of constant width. The purpose of this section is to properly assemble all these pieces to finally obtain the boundary of a $4$-dimensional convex body of constant width.
We recommend consulting the paper \cite{AMO},  because many of the ideas and techniques developed below have their origin and  are explained in greater detail there.
\smallskip

For simplicity, as we have done several times before, we are assuming that $(E, \mathcal{H})$ is the  pair of the focal quadrics defined in \eqref{EH} and let  $\phi_1({\x},{\y})$, $\phi_2({\x},{\y})$, $L({\x},{\y})$, $R_{\x}$  and $R_{\y}$ be as in Section \ref{sec:distance}. 
Let us consider the restriction  $\phi_1\colon \mathcal{H}_{345}\times E_{12} \to \mathbb{R}^4$ and 
$\phi_2\colon \mathcal{H}_{345}\times E_{12} \to \mathbb{R}^4$. Note that $\phi_1(p_i,{\y})=p_i$  for every ${\y}\in E_{12}$, $i=3,4,5$,  
and similarly, $\phi_2({\x},p_i)=p_i$  for every 
${\x}\in \mathcal{H}_{345}, i=1,2$.  In any other case,  
\[\phi_1\colon \mathcal{H}_{345}\times E_{12} \to \mathbb{R}^4\text{\ and\ }
\phi_2\colon \mathcal{H}_{345}\times E_{12} \to \mathbb{R}^4\text{\ are smooth embeddings.}\]

\begin{figure}[h]

\begin{tikzpicture}

\node at (0,0){\includegraphics[scale=.35]{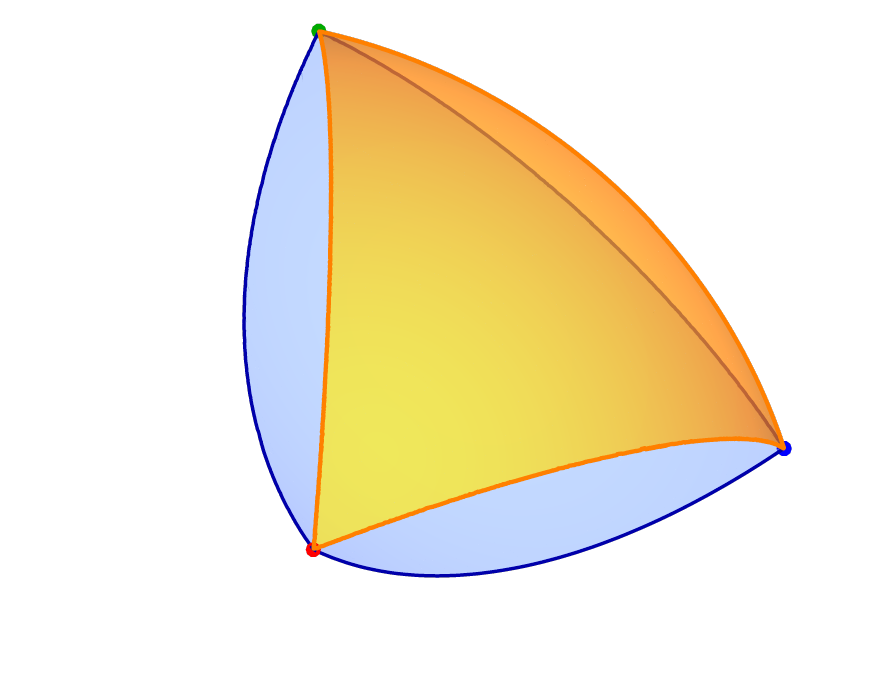}};
\node at (6,-5){\includegraphics[scale=.35]{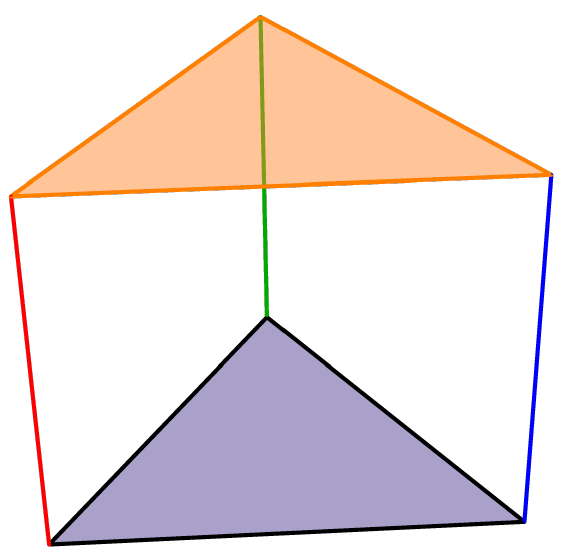}};
\node at (12,.25){\includegraphics[scale=.5]{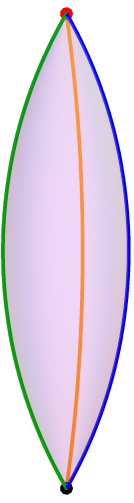}};
\node[rotate=-90] at (6.5,0){\includegraphics[scale=.45]{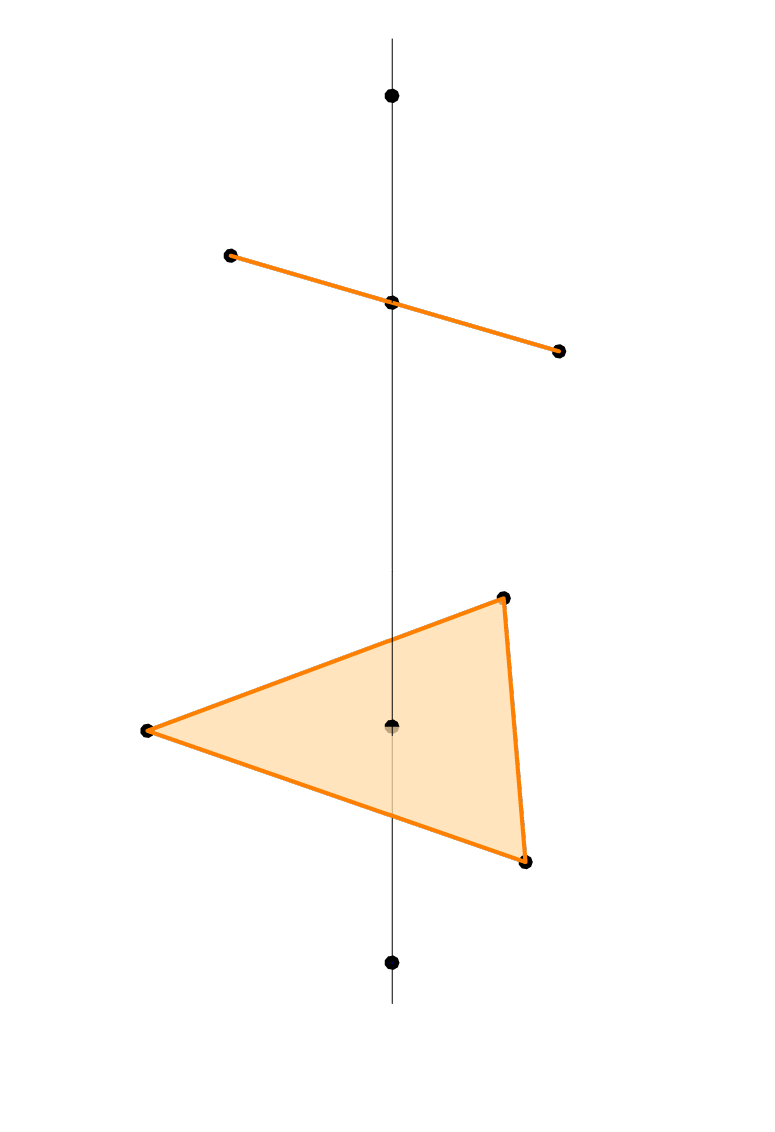}};
\path [-latex] (4,-5.5) edge [bend left] (1,-1.5);
\path [-latex] (8,-5.5) edge [bend right] (11.75,-2.5);
\node at (9,1.45){$P_1$};
\node at (8,-1.6){$P_2$};
\node at (5.5,2){$P_3$};
\node at (6.5,-1.1){$P_4$};
\node at (4,-1.2){$P_5$};
\node at (3.8,.25){$\phi_1(x,y)$};
\node at (10,.25){$\phi_2(x,y)$};
\node at (5.3,-.25){$x$};
\node at (8.3,-.25){$y$};
\node at (2,-4.5){$\phi_1$};
\node at (10.5,-4.8){$\phi_2$};
\node at (.75,1.6){\textcolor{orange}{$B_{P_1}$}};
\node at (-1.5,-.3){\textcolor{blue}{$B_{P_2}$}};
\node at (-.75,2.1){$\phi_1(P_3,y)$};
\node at (2.5,-1){$\phi_1(P_4,y)$};
\node at (-.6,-1.7){$\phi_1(P_5,y)$};
\node at (11.15,1){\textcolor{olive}{$B_{P_3}$}};
\node at (12.9,.5){\textcolor{blue}{$B_{P_4}$}};
\node at (11.8,0){\textcolor{orange}{$B_{P_5}$}};
\node at (12,-2){$\phi_2(x,P_1)$};
\node at (12,2.5){$\phi_1(x,P_2)$};
\node at (6,-6.8){$T$};
\node at (8.4,-1){\textcolor{orange}{$I$}};
\node at (5.25,-1.25){\textcolor{orange}{$T$}};
\node at (6,-3.2){$P_3P_2$};
\node at (6.4,-5.1){$P_3P_1$};
\node at (7.75,-6.75){$P_4P_1$};
\node at (8.1,-4.4){$P_4P_2$};
\node at (4.4,-6.9){$P_5P_1$};
\node at (3.9,-4.5){$P_5P_2$};
\end{tikzpicture}
\caption{The wedge-manifolds }\label{wedge}

\end{figure}
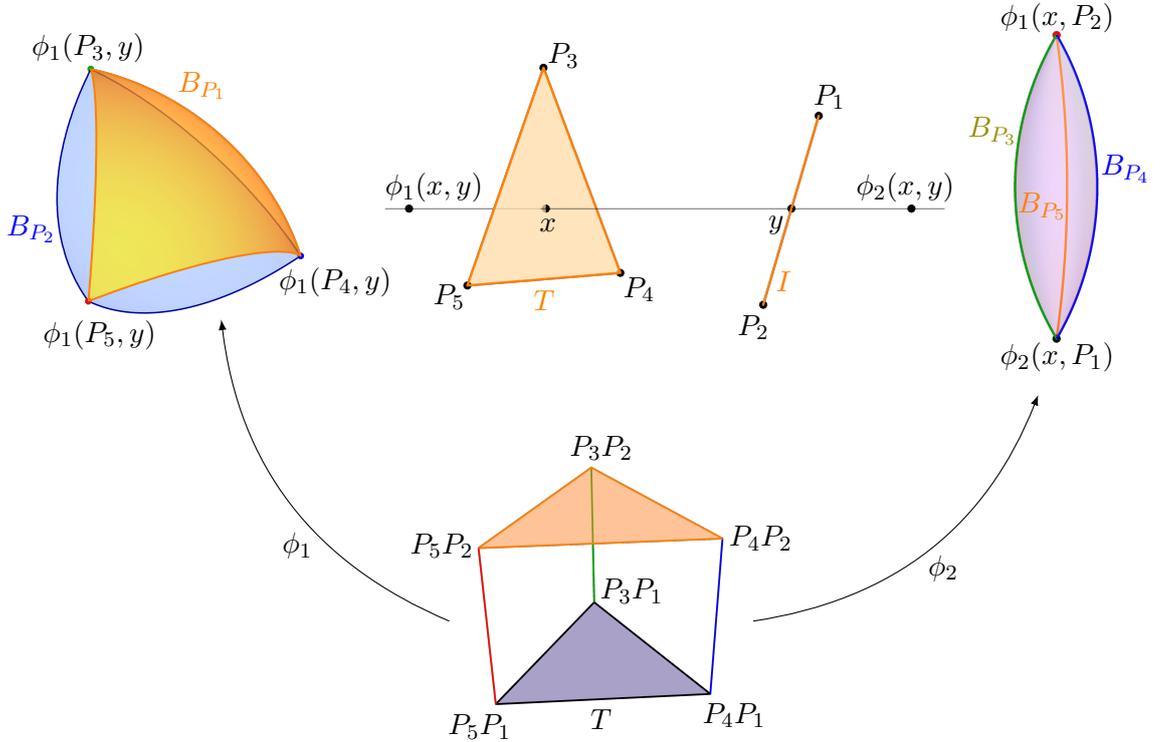

Indeed, $\phi_1 (\mathcal{H}_{345}\times E_{12})$ is a $3$-dimensional  topological ball with the shape of a ``fat triangle" whose boundary 
is made of two topological triangles $\phi_1 (\mathcal{H}_{345}\times \{p_1\})$ and $\phi_1 (\mathcal{H}_{345} \times \{p_2\}$) and three wedge surfaces 
($\phi_1(E_{34}\times E_{12})$, $\phi_1(E_{45}\times E_{12})$ and $\phi_1(E_{35}\times E_{12})$) whose boundary consists of two arcs of the form
$\phi(E_{jk}\times p_i)$, where $j,k\subset \{3,4,5\}$ and $i=1, 2$.

Similarly,  $\phi_2 (\mathcal{H}_{345}\times E_{12})$ is a $3$-dimensional  topological ball whose boundary 
is made of 3 wedge surfaces  ($\phi_2(E_{34}\times E_{12})$, $\phi_2(E_{45}\times E_{12})$ and $\phi_2(E_{35}\times E_{12})$) whose boundary 
consists of 3 arcs of the form
$\phi_2(p_i\times E_{12})$, $i\subset \{3,4,5\}$. See Figure \ref{wedge}. 

\smallskip
Denote by $(345)=\phi_1 (\mathcal{H}_{345}\times E_{12})\subset \R^4$ and by $(12)=  \phi_2 (\mathcal{H}_{345}\times E_{12})\subset \R^4$. These will be called the \emph{wedge-submanifolds} corresponding to the hyperbolic triangle $T_{345}$ and the elliptic edge $E_{12}$, respectively.

Note that $\phi_1 (\mathcal{H}_{345}\times \{p_i\})\subset \mathbb{S}^3_{p_1}$, $i=1,2$ and $\phi_2(p_i\times E_{12})\subset \mathbb{S}^3_{p_i}$, $i\subset \{3,4,5\}$,
where $\mathbb{S}^3_{p_i}$, denotes the $3$-dimensional sphere with center at $p_i$ and radius $2z_1$, $i=1,\dots,5$.
We shall denote by $(345)_1=\phi_1 (\mathcal{H}_{345}\times \{p_1\})$ the embedded triangle contained in $\mathbb{S}^3_{p_1}$ and by $(345)_2=\phi_1 (\mathcal{H}_{345}\times \{p_2\})$ the embedded triangle contained in $\mathbb{S}^3_{p_2}$.

In summary, we have the following lemma:

\begin{lemma}\label{lemY}
For every point $u\in (345)$, there is a point $v\in (12)$ such that $uv=2z_1$, and there is a  hyperplane tangent to $(12)$ at $v$ and a  hyperplane tangent to $(345)$ at $u$ orthogonal to  $uv$. Furthermore, if  $u\in (12)\setminus \{p_1,p_2\}$, then the tangent hyperplane of $(12)$ at $u$ is unique; similarly, if $v\in (345)\setminus \{p_3,p_4, p_5\}$, then the tangent hyperplane of $(345)$ at $v$ is unique. Moreover,  $d\big((345),(12)\big)=2z_1$. 
\end{lemma}

Using the focal $2$-skeleton of  $\Delta^4=\{p_1, p_2, p_3, p_4, p_5\}$ described at the end of Section \ref{sec:2-skeleton}, we are now able to obtain,
for every triple $\{p_i, p_j, p_k\}\subset \{p_1, p_2, p_3, p_4, p_5\}$, a wedge-submanifold $(ijk)$  and for a pair 
$\{p_i, p_j\}\subset \{p_1, p_2, p_3, p_4, p_5\}$ a wedge-submanifold $(ij)$.  Our next step is to show that $(345)$ and $(45)$ assemble correctly because 
the restrictions of $\phi_1\colon (E_{45}\times E_{12})\to \R^4$ and $\phi'_2\colon (E_{12}\times E_{45})\to\R^4$ coincide, where 
\[\phi_1:(\mathcal{H}_{345}\times E_{12})\to \R^4\]
is defined at the end of Section \ref{sec:distance}, and 
\[\phi'_2:(\mathcal{H}_{123}\times E_{45})\to\R^4\]
is the corresponding embedding for the focal pair $(\mathcal{H}_{123}, E_{45})$ to obtain $(45)$
(Figure \ref{assemble}).

\begin{figure}[h!]
\begin{center}
\begin{tikzpicture}
\node at (-.25,0){\includegraphics[scale=.3]{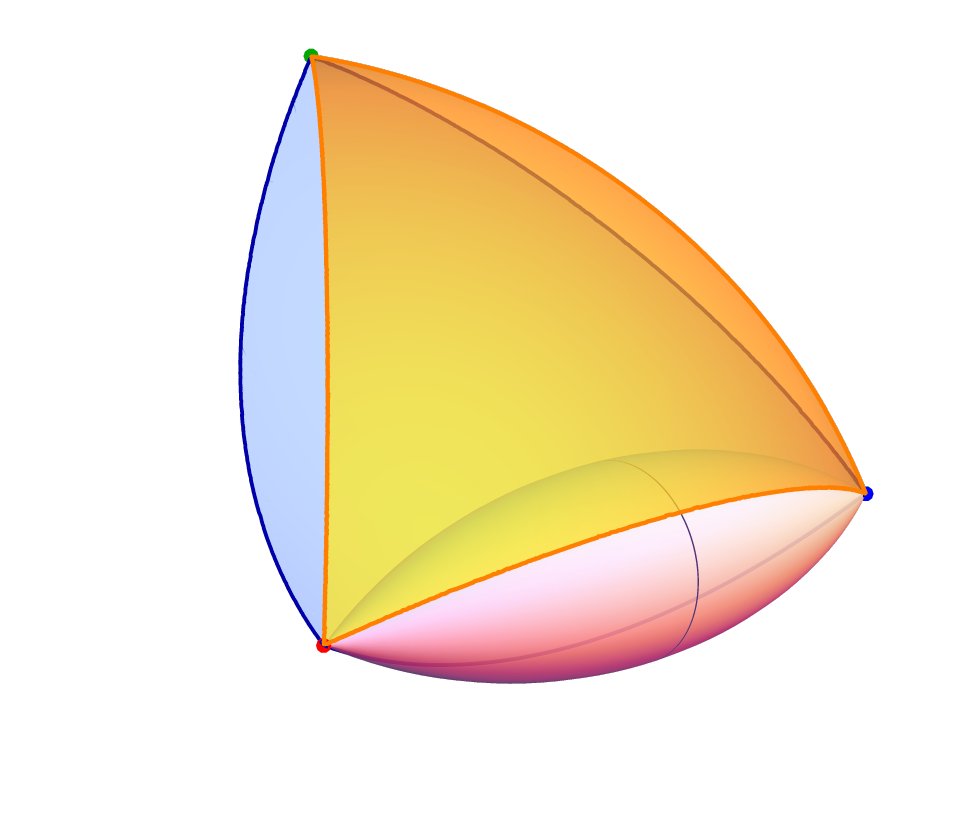}};
\node at (0,-6){\includegraphics[scale=.35]{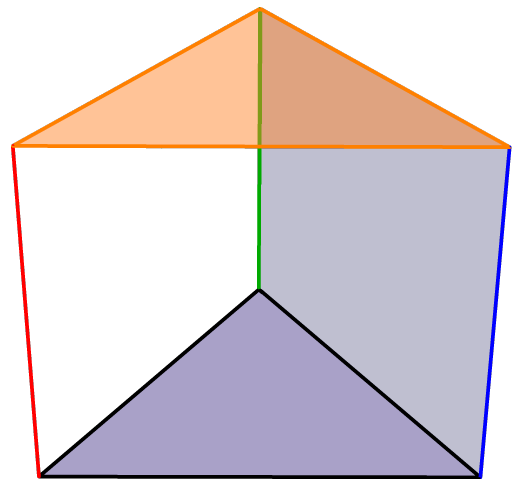}};
\node at (6,-4){\includegraphics[scale=.35]{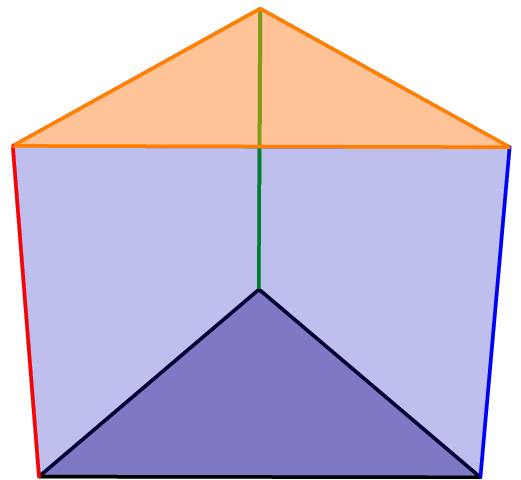}};
\path [-latex] (4,-4.5) edge [bend left] (1.5,-1.5);
\node at (2,-3.5){$\phi_1$};
\node at (-1,-3){$\phi_2$};
\path [-latex] (-.75,-4.5) edge [bend left] (0,-1.85);
\node at (-1,-1.5){$P_4$};
\node at (1.8,-.75){$P_5$};
\node at (-1,2){$P_3$};
\node at (1,1.25){$(345)$};
\node at (.6,-1.7){$(45)$};
\node at (4,-3.45){$P_5P_2$};
\node at (4.25,-5.5){$P_5P_1$};
\node at (8,-3.45){$P_4P_2$};
\node at (7.75,-5.5){$P_4P_1$};
\node at (6,-2.4){$P_3P_2$};
\node at (6.5,-4.25){$P_3P_1$};
\node at (0.1,-4.4){$P_2P_4$};
\node at (0.5,-6.25){$P_2P_5$};
\node at (2,-5.5){$P_1P_4$};
\node at (1.75,-7.5){$P_1P_5$};
\node at (-1.75,-7.5){$P_3P_5$};
\node at (-2,-5.5){$P_3P_4$};
\end{tikzpicture}
\caption{The assembly of two wedge-submanifolds }\label{assemble}
\end{center}
\end{figure}
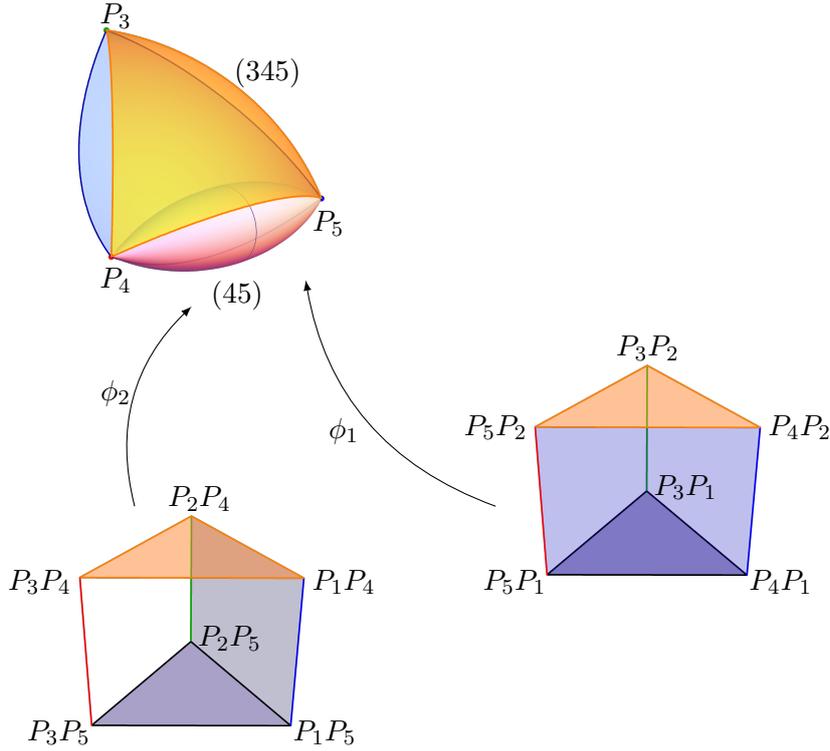

By Lemma \ref{milagro}, the arc of ellipse $E_{45}$ is a subset of the hyperboloid of revolution $\mathcal{H}$. 
Then we know that the collection of 4-balls given by $\cup_{{\x}\in E_{45}} B'_{\x}$ comes from the Steiner chain of disks $\Omega_{E_{45}}$ with centers in $E_{45}$ and with radius $R'_{\x}$ (for every $x\in E_{45}$). But at the same time, since $E_{45}\subset \mathcal{H}$, we have another collection of 4-balls $\cup_{{\x}\in E_{45}} \mathcal{B}_{\x}$ that comes from the Steiner chain of $4$-dimensional balls $\Omega_{\mathcal{H}}$ with center in the hyperboloid $\mathcal{H}$ and radius $\mathcal{R}_{\x}$.  We want to show that   $R'_{\x}= \mathcal{R}_{\x}$ and hence $B'_{\x}=\mathcal{B}_{\x}$ for every ${\x}\in E_{45}$.

\begin{lemma}
\[R'_{\x}=\mathcal{R}_{\x} \quad \mbox { and consequently } \quad B'_{\x}=\mathcal{B}_{\x}.\]
\end{lemma}

\proof  Let ${\x}\in E_{45}$ and ${\y}\in E_{12}$. Denote by $\phi_1({\x},{\y})$ the point in $L({\x},{\y})$ at a distance $\mathcal{R}_{\x}$ from ${\x}$, where ${\x}$ is between $\phi_1({\x},{\y})$ and ${\y}$ and let 
  $\phi'_2({\y},{\x})$ be the point in $L({\y},{\x})$ at a distance $R'_{\x}$ from ${\x}$, with ${\x}$ between ${\y}$ and $\phi'_2({\y},{\x})$,  As above, ${\x}{\y}+\mathcal{R}'_{\y}+R'_{\x}=2z_1={\x}{\y}+\mathcal{R}_{\x}+R_{\y}$, and by Lemma \ref{lemY},  there is a  hyperplane tangent to (345) at  $\phi_1({\x},{\y})$ and a  hyperplane tangent to  $(45)$ at  $\phi'_2({\y},{\x})$, both orthogonal to $L({\x},{\y})$. This implies that for every   ${\x}\in E_{45}$ and every ${\y}\in E_{12}$, $\phi_1({\x},{\y})\phi'_2({\y},{\x})$ is a constant independent of ${\x}$ and ${\y}$ and hence, since $\phi_1(p_4,{\y})=\phi'_2(\y, p_4)=p_4$, we have that $\phi_1({\x},{\y})=\phi'_2({\y},{\x})$ and therefore $R'_{\x}=\mathcal{R}_{\x}$, as we wished.  \qed

\medskip

As a consequence of the above lemma, $\phi_1(E_{45}\times E_{12})=\phi'_2(E_{12}\times E_{45})$ and hence we can assemble $(45)$ and $(345)$ as shown in Figure \ref{assemble}.  Indeed, in a similar way, this allows us to assemble $(ij)$ and $(ijk)$ when $\{i,j\}\subset \{i,j,k\}$ and 
$\{i,j,k\}\subset \{1,\dots,5\}$.  Figure \ref{assemble2} shows the assembly of $(345)$, $(45)$ and $(145)$.  Note that $(45)$  is assembled  with exactly 3 of these ``fat triangles," namely $(345), (145)$ and $(245)$. 

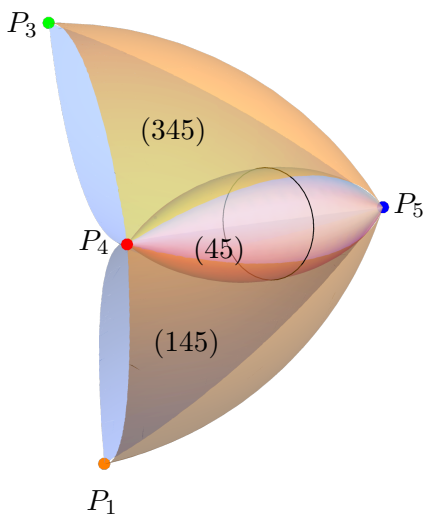
\begin{figure}[h]
\begin{center}
\includegraphics[scale=.6]{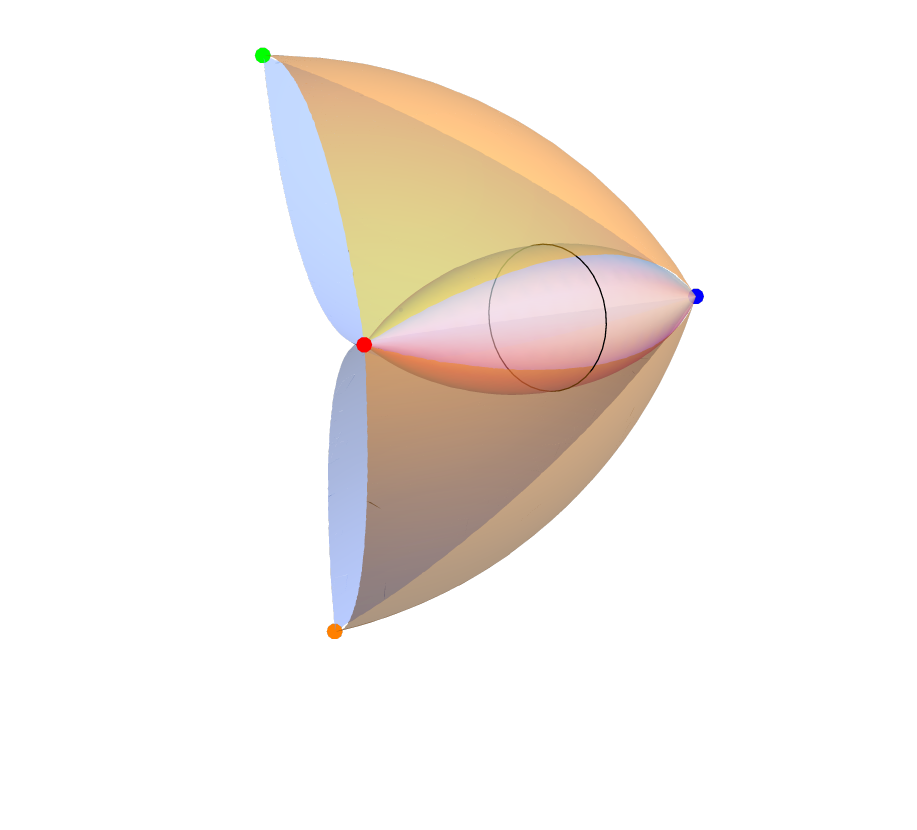}\put(-175,40){$P_1$}\put(-205,220){$P_3$}\put(-178,138){$P_4$}\put(-60,152){$P_5$}\put(-155,180){$(345)$}\put(-150,100){$(145)$}\put(-135,135){$(45)$}
\caption{The assembly of $(345)$, $(45)$ and $(145)$ }\label{assemble2}
\end{center}
\end{figure}
\bigskip

In the $3$-sphere  $\mathbb{S}^3_{p_1}$ with radius $2z_1$ and center at $p_1$, we have the embeddings of the four topological triangles  $(234)_1$, $(245)_1$, $(235)_1$ and $(345)_1$.  We shall show next that their relative interiors are pairwise disjoint in such a way that all these triangles give rise to the embedding of a topological $2$-sphere which bounds a $3$-manifold in $\mathbb{S}^3_{p_1}$. By symmetry it will be enough to prove Lemma \ref{lemQ}.  

Recall that a hyperplane $\Gamma$ separates $A$ and $B$ if $A\subset \Gamma^+$ and $B\subset \Gamma^-$, where $\Gamma^+$ and $\Gamma^-$ are the two half-spaces determined by 
$\Gamma$.

\begin{lemma}\label{lemQ}
There is a hyperplane $\Gamma\subset \R^4$ that contains the common face of $(345)_1$ and $(245)_1$ but separates the two triangles.
\end{lemma}

\proof  As before, we shall work in euclidean $4$-space $\R^4$ with coordinates $\{x,y,z,w\}$. 
Let us consider the $\{x,y,w\}$ space that contains the hyperboloid $\mathcal{H}_{345}$. As in Section \ref{subsec:focallyembedded}, denote by $\Gamma^2$ the $2$-dimensional plane through $p_4p_5$ and the barycenter $0$ of the $4$-simplex $ \Delta^4=\{p_1, p_2, p_3, p_4, p_5\}$.  In the $\{x,y,w\}$-space,
the arc of ellipse $E_{45}=E' \subset \mathcal{H}_{345}$ is contained in $\Gamma^2$, and furthermore $\Gamma^2 \cap \mathcal{H}_{345}=E_{45}$.

Let $\Gamma$ be the hyperplane generated by $\Gamma^2$ and $p_1$ and note that there is an orthogonal reflection $\rho\colon \R^4\to \R^4 $ that preserves the hyperplane $\Gamma$ and of course interchanges $p_2$ with $p_3$. Since this reflection $\rho$ is a symmetry of $\Delta^4$, by Section \ref{sec:2-skeleton}, it is also a symmetry of the focal $2$-skeleton of  $\Delta^4$. Consequently $\rho(\mathcal{H}_{345})= \mathcal{H}_{245}$, and therefore $\Gamma \cap \mathcal{H}_{345}=e_{45}= \Gamma \cap \mathcal{H}_{345}.$

Finally, note that by construction, $(345)_1$ is the intersection with $\mathbb{S}^3_{p_1}$ of the cone of $\mathcal{H}_{345}$  with apex $p_1$. Similarly, $(245)_1$ is the intersection with $\mathbb{S}^3_{p_1}$ of the cone of $\mathcal{H}_{245}$  with apex $p_1$. Consequently, the common face of $(345)_1$ and $(245)_1$ is contained in $\Gamma$, but it separates $(345)_1$ and $(245)_1$, as we wished.\qed

\medskip

As a consequence of Lemma \ref{lemQ},  the embeddings of the four topological triangles  $(234)_1$, $(245)_1$, $(235)_1$ and $(345)_1$ in  $\mathbb S^3_{p_1}$ give rise to the embedding of a topological $2$-sphere which bounds a topological $3$-manifold. We shall denote this 
wedge-submanifold by $(2345)$. Similarly we obtain $(1345)$, $(1245)$, $(1235)$ and $( 1234)$.  These $5$ wedge-submanifolds are the $3$-sphere caps of our construction. 

\medskip

For every subset 
$A$ of $\{1,\dots,5\}$, of size $1\leq |A|<5$, denote the wedge-submanifold $(A)$,  Then  
\[\Theta=\bigcup_{1\leq |A|<5} (A) \subset \R^4.\]

Note that for every given face of a wedge-submanifold $(A)$ there is precisely one other face of a different wedge-submanifold $(B)$ that coincides with it. This proves that the union of all  
the wedge-submanifold assembles perfectly; furthermore, the pieces do not intersect in the interior. Then $\Theta$ is the image of a smooth submersion 
\[ \xi\colon  M^3\to \R^4,\]
where $M^3$ is a close $3$-dimensional manifold and $\xi(M^3)=\Theta$.

As a consequence of Lemma \ref{lemY}, the following lemma emphasizes the fundamental properties of the smooth submersion 
$ \xi\colon M^3 \to \R^4$.

\begin{lemma}\label{lemnormal}
For every point ${\x}\in M^3$,  there is a unique point ${\y}\in M^3$ such that the length of $\xi({\x})\xi({\y})$ is $2z_1$ and 
the planes orthogonal to $\xi({\x})\xi({\y})$ at $\xi({\x})$ and $\xi({\y})$ are tangent planes of the smooth submersion $\Theta$.  Furthermore, if  
$\xi({\x})$ is different from $\{p_1, p_2, p_3, p_4, p_5\}$, the tangent plane of the submersion $\Theta$ at $\xi({\x})$ is unique. 
\end{lemma}

Our last task is now to show that $\Theta$ is the boundary of a $4$-dimensional body of constant width, but it will still take us some effort to formalize a proof.  In particular, it will be important to prove that its diameter is $2z_1$. Our next section is devoted to proving this fact.

\section{The Diameter of  the Submersion \boldmath{$\Theta $} }\label{sec:diameter}

In this section we would like to prove that $\Theta$  is the boundary  of a $4$-dimensional body of constant width. For this purpose, it is essential to prove that the diameter of $\Theta$ is $2z_1$.
\bigskip

We begin by considering the collection of lines $L(\x,\y)$, from a point $\x$ of an arc of an ellipse to point $\y$ of a piece of a hyperboloid of revolution, where the arc of the ellipse and the piece of the hyperbololoid of revolution are members of the focal $2$-skeleton of $\Delta^4=\{p_1, p_2, p_3, p_4, p_5\}$.
\[L_{ij}=\{L({\x},{\y})\mid {\x}\in E_{ij}\setminus \{p_i,p_j\}\mbox{ and } 
{\y}\in \mathcal{H}_{k\lambda\alpha}\setminus \{p_k,p_\lambda,p_\alpha\},\]
where $\{i,j,k,\lambda,\alpha\}=\{1,2,3,4,5\}$.

\begin{lemma}\label{disjuntos}
If $\{i,j\}\not=\{k,\lambda\}$, then $L_{ij}\cap L_{k\lambda}=\emptyset.$
\end{lemma}

\proof Let $\Upsilon$ be the set of lines in $\R^4$ except those lines parallel to a face of $\Delta_4$.
We start by proving that $L_{ij}\subset\Upsilon$.   It is not difficult to see that if $\Gamma$ is a hyperplane through a facet of $\Delta^4$, then $\Gamma$ separates $E_{ij}$ from  $\mathcal{H}_{k\lambda\alpha}$.  Indeed, 
$\Gamma \cap (E_{ij}\cup \mathcal{H}_{k\lambda\alpha})\subset \Delta_4$. Consequently, every plane parallel to $\Gamma$ but different from it is not  transversal to $E_{ij}$ and  $\mathcal{H}_{k\lambda\alpha}$. Suppose now there is a line $L\in L_{ij}$ which is parallel to a face of $\Delta^4$.
First, by the definition of $L$ we know that it is not contained in $\Gamma$. Let $\Gamma'$ be a hyperplane containing $L$ and parallel to this face. Then $\Gamma'$ is transversal to  $E_{ij}$ and  
$\mathcal{H}_{k\lambda\alpha}$, which implies that $\Gamma'$ contains a face of $\Delta^4$, which is a contradiction. 

Note now that $\Upsilon$ has many connected components. Indeed, $10$ of them correspond to those lines $L$ with the property that the orthogonal projection of $\Delta_4$ onto $L^\perp$ maps the vertices of $\Delta^4$
onto five vertices in general and convex position. For every $\{i,j\}\subset\{1,\dots,5\}$, let $\ell_{ij}$ be the line through the midpoint of $p_ip_j$ and the barycenter of the triangle with vertices $\{p_k, p_\lambda, p_\alpha\}$. Remember that 
$\ell_{ij}\in L_{ij}$ because $\ell_{ij}$ is the common axis of the ellipse $E_{ij}$ and the hyperboloid of revolution $\mathcal{H}_{k\lambda\alpha}$. Hence, every $\ell_{ij}$ is in a different component of $\Upsilon$. The lemma now follows from the fact that every $L_{ij}$ is connected as a topological subspace of $\Upsilon$ and $\ell_{ij}\in L_{ij}$. \qed

\medskip
Denote by $\mathbb T$ the Reuleaux $5$-simplex $\mathbb T=\bigcap_{\{1,\dots 5\}} B_{p_i},$ where $B_{p_i}$ is the $4$-ball  of radius $2z_1$ and center at $p_i$, $1\leq i\leq 5$.


\begin{lemma}\label{reuleaux}
$\Theta$ is contained in the Reuleaux $5$-simplex $\mathbb T$.
\end{lemma}
\proof We begin by proving that the set $(345)\cup \Delta_4$ is contained in the Reuleaux $5$-simplex $\mathbb T$.
Note that every diameter is a binormal; hence, it is easy to observe by the construction of $(345)$ that the binormals of $(345)\cup \Delta^4$ are of the form $p_ip_j$. Consequently, the diameter of  $(345)\cup \Delta_4$ is $2z_1$ and hence 
$(345)\cup \Delta_4\subset B_{p_i}$, for $i=1,2$. So,  $(345)\cup \Delta_4 \subset \mathbb T$.  Similarly, for every triple $\{i,j,k\}\subset\Delta_4$, we have that $(ijk)\cup \Delta_4\subset \mathbb T$.

By the same arguments, every diameter of $(12)\cup  \Delta^4$  is of the form $p_ip_j$. Consequently, the diameter of  
$(1,2)\cup \Delta_4$ is $2z_1$ and hence $(12)\cup \Delta^4\subset B_{p_i}$, for $i=3,4,5$. So,  $(12)\cup \Delta^4\subset \mathbb T$. 
Similarly, for every pair $\{i,j\}\subset\Delta_4$, we have that $(ij)\cup \Delta^4\subset \mathbb T$.

Consider now the wedge-submanifold $(2345)\subset \mathbb{S}^3_{p_1}\subset B_{p_1}$. By the above, its boundary 
$\partial (2345)=(234)_1\cup (235)_1\cup (245)_1\cup (345)_1$ is contained in $B_{p_2}$.  Therefore, $\partial (2345)$ is contained in the convex spherical cap $B_{p_2}\cap \mathbb{S}^3_{p_1}$ and hence, clearly $(2345)\subset B_{p_2}\cap \mathbb{S}^3_{p_1}\subset B_{p_2}$.  Similarly $(2345)\subset B_{p_i}$, for $2,\dots,5$. Therefore $(2345)\subset \mathbb T$.
The same holds for every $3$-spherical cap of $\Theta$. This concludes the proof of Lemma \ref{reuleaux}.\qed


\begin{lemma}\label{diameter}
The diameter of $\Theta$ is $2z_1$.
\end{lemma}
\proof  Let ${\x}{\y}$ be a diameter of $\Theta$ and let $L$ be the line through ${\x}$ and ${\y}$. By Lemma \ref{reuleaux},  
the chord ${\x}{\y}$ is contained in the Reuleaux $5$-simplex $\mathbb T$.  If ${\x}=p_i$, since every chord of $\mathbb T$ with endpoint $p_i$ has length $2z_1$, we conclude that the length of ${\x}{\y}$ is $2z_1$, as we wished.  By the same argument, we may assume that
$L\cap  \Delta^4 = \emptyset $.

If $L\cap  \Delta^4=\emptyset$, by Lemma \ref{lemnormal}, the normal line to 
$\theta$ at ${\x}$ is unique. Therefore, the line $L$, as a normal of $\Theta$ at ${\x}$, is a member of $L_{ij}$. If the length of ${\x}{\y}$ is greater than $2z_1$,  the line $L$, as a normal of $\Theta$ at ${\y}$, is a member of $L_{k\lambda}$, where $\{i,j\}\not=\{k,\lambda\}$. But this is impossible, by Lemma \ref{disjuntos}.  Therefore the length of ${\x}{\y}$ is $2z_1$ and  the diameter of $\Theta$ is $2z_1$. \qed


\begin{theorem}
 $\Theta$ is the boundary of a $4$-dimensional body of constant width.   
\end{theorem}\label{principal}

\proof  Consider the convex hull $\conv(\Theta)$ of $\Theta$. By Lemma \ref{diameter}, the diameter of $\conv(\Theta)$ is $2z_1$. Moreover by Lemma \ref{lemnormal}, every point of $\Theta$ is the endpoint of a diameter of $\conv(\Theta)$. This implies immediately that $\Theta$ is the boundary of $\conv(\Theta)$. Moreover, again by Lemma \ref{lemnormal}, $\Theta$ is smooth, with the exception of $\{p_1, p_2, p_3, p_4, p_5\}$, at which points $\Theta$  has vertex singularities. Consequently, every normal of $\conv(\Theta)$ through a non-vertex point  
is a binormal of length $2z_1$. This implies, by \cite[Theorem 3.1.7]{MMO}, that $\conv(\Theta)$ is a body of constant width. \qed

\medskip

\noindent {\bf Acknowledgments.} Luis Montejano and Deborah Oliveros acknowledge  support  from  PAPIIT-UNAM under project 35-IN112124.

\end{document}